\documentclass[11pt, reqno, oneside]{article}
\usepackage{geometry} % to change the page dimensions
\usepackage[affil-it]{authblk}

\usepackage{amsmath, amsthm, amsfonts, amssymb, amsbsy, bigstrut, graphicx, enumerate,  upref, longtable, comment, booktabs, array, caption, subcaption}

\usepackage{times}

\usepackage[breaklinks]{hyperref}

\usepackage[numbers, sort&compress]{natbib}

\newtheorem{thm}{Theorem}[section]
\newtheorem{lmm}[thm]{Lemma}
\newtheorem{cor}[thm]{Corollary}

\newtheorem{defn}[thm]{Definition}

\theoremstyle{definition}

\newcommand{\bigavg}[1]{\biggl\langle #1 \biggr\rangle}

\newcommand{\ee}{\mathbb{E}}

\newcommand{\rr}{\mathbb{R}}
\newcommand{\smallavg}[1]{\langle #1 \rangle}

\newcommand{\var}{\mathrm{Var}}
\newcommand{\ve}{\varepsilon}

\newcommand{\fpar}[2]{\frac{\partial #1}{\partial #2}}

\newcommand{\mpar}[3]{\frac{\partial^2 #1}{\partial #2 \partial #3}}

\numberwithin{equation}{section}

\usepackage[utf8]{inputenc}
\usepackage{amsmath}
\usepackage{amsfonts}
\usepackage{bbm}
\usepackage{mathtools}
\usepackage{tikz}
\usetikzlibrary{decorations.pathreplacing}
\usepackage{amsthm}
\usepackage[english]{babel}
\usepackage{fancyhdr}
\usepackage{amssymb}
\usepackage{enumitem}
\usepackage{qtree}
\usepackage{mathrsfs}

\newcommand{\Z}{\mathbb{Z}}

\newcommand{\R}{\mathbb{R}}

\newcommand{\E}{\mathbb{E}}

\renewcommand{\P}{\mathbb{P}}

\newcommand{\sij}{\smallavg{\sigma_i\sigma_j}}

\newcommand{\si}{\smallavg{\sigma_i}}
\newcommand{\sj}{\smallavg{\sigma_j}}

\begin{document}
%\title{Local KPZ growth in a class of random surfaces}
%[Recent developments in measures of association]
%\title{A short-range model with the properties of a mean-field spin glass}
%\title{The random field Ising model has a spin glass phase}
\title{Features of a spin glass in the random field Ising model}
\author{Sourav Chatterjee\thanks{Department of Statistics, Stanford University, 390 Jane Stanford Way, Stanford, CA 94305, USA. Email: \href{mailto:souravc@stanford.edu}{\tt souravc@stanford.edu}. 
}}
\affil{Stanford University}
%The author was partially supported by NSF grants DMS-2113242 and DMS-2153654. The author thanks xxx for helpful comments.

%\address{Departments of mathematics and statistics, Stanford University}
%\email{souravc@stanford.edu}
%\dedicatory
%Research partially supported by NSF grants DMS-1855484 and DMS-2113242}
%\keywords{}
%\subjclass[2020]{}

\maketitle

%\begin{center}
%{\it \small In honor of friend and teacher Prof.~Rajeeva L.~Karandikar on the occasion of his 65$^{th}$ birthday.}
%\end{center}

\begin{abstract}
A longstanding open question in the theory of disordered systems is whether short-range models, such as the random field Ising model or the Edwards--Anderson model, can indeed have the famous properties that characterize mean-field spin glasses at nonzero temperature. This article shows that this is at least partially  possible in the case of the random field Ising model. Consider the Ising model on a discrete $d$-dimensional cube under free boundary condition, subjected to a very weak i.i.d.~random external field, where the field strength is inversely proportional to the square-root of the number of sites. It turns out that in $d\ge 2$ and at subcritical temperatures, this model has some of the key features of a mean-field spin glass. Namely, (a) the site overlap  exhibits one step of replica symmetry breaking, (b) the quenched distribution of the overlap is non-self-averaging, and (c) the overlap has the Parisi ultrametric property. Furthermore, it is shown that for Gaussian disorder, replica symmetry does not break if the field strength is taken to be stronger than the one prescribed above, and non-self-averaging fails if it is weaker, showing that the above order of field strength is the only one that allows all three properties to hold. However, the model does not have two other features of mean-field models. Namely, (a) it does not satisfy the Ghirlanda--Guerra identities, and (b) it has only two pure states instead of many. %An interesting consequence of the proof is a new ``uniformity of correlations'' result for the ordinary Ising model. Two proofs of this re%, as in mean-field models. %In the process, a general definition of the number of pure states in a sequence of finite volume systems is given, filling a gap in the literature.%The proof techniques yield a number of interesting byproducts for the ferromagnetic and antiferromagnetic Ising models in general dimensions.
%, and (d) the overlap has the Parisi ultrametric property
\end{abstract}

%\tableofcontents
%\section{A new coefficient of correlation}
\section{Introduction}
%We begin with the main results, followed by a discussion of the related literature and motivation. Proofs are in later sections.
The random field Ising model (RFIM) was introduced as a simple model of a disordered system by \citet{imryma75} in 1975. The model is defined as follows. Take any $d\ge 1$ and $\Lambda \subseteq \Z^d$. Let $E$ denote the set of edges connecting neighboring points in $\Lambda$. Given a field strength $h\in \R$, define the (random) Hamiltonian $H:\{-1,1\}^\Lambda \to \R$ as
\begin{align}\label{original}
H(\sigma) := -\sum_{\{i,j\}\in E} \sigma_i\sigma_j - h \sum_{i\in \Lambda} J_i\sigma_i,
\end{align}
where $J = (J_i)_{i\in \Lambda}$ is a fixed realization of i.i.d.~random variables from some distribution. At inverse temperature $\beta > 0$, the RFIM prescribes a random Gibbs measure on $\{-1,1\}^\Lambda$ with probability mass function proportional to $e^{-\beta H(\sigma)}$. 

A large body of deep mathematics has grown around this model, such as the early works of \citet{imbrie84, imbrie85} on the multiplicity of ground states in the 3D RFIM, the proof of phase transition in $d\ge 3$ by \citet{bricmontkupiainen87, bricmontkupiainen88}, the absence of phase transition in $d\le 2$ proved by \citet{aizenmanwehr89, aizenmanwehr90}, and the more recent works on  quantifying the Aizenman--Wehr theorem~\cite{chatterjee18, camiaetal18, aizenmanpeled19, aizenmanetal20}, culminating in the proof of exponential decay of correlations in the 2D RFIM by \citet{dingxia21}. The recent developments have led to a resurgence of interest in this model in the mathematical community, yielding a number of new and important results~\cite{dingzhuang21, dingwirth23, darioetal21, bowditchsun22, dingetal22, barnir22}.

In spite of all this progress, one major question that has not yet been settled is whether the RFIM has a spin glass phase. A disordered system is said to exhibit spin glass behavior if it has the properties that characterize mean-field spin glasses. In the formulation laid out by Giorgio Parisi~\cite{parisi07}, the main features of mean-field spin glasses are replica symmetry breaking (RSB), non-self-averaging (NSA), ultrametricity, and the presence of many pure states. These properties are defined as follows. Consider a system of $N$ particles, with spins $\sigma = (\sigma_1,\ldots, \sigma_N)\in \{-1,1\}^N$. In a disordered system, the probability law $\mu$ of $\sigma$ is random. Let $\sigma^1,\sigma^2,\ldots$ be i.i.d.~spin configurations drawn from a fixed realization of the random probability measure $\mu$. The overlap between the configurations $\sigma^i$ and $\sigma^j$ is defined as
\[
R_{i,j} := \frac{1}{N}\sum_{k=1}^N \sigma^i_k\sigma^j_k. 
\]
Let $\smallavg{R_{1,2}}$ denote the expected value of $R_{1,2}$ with respect to $\mu$. Roughly speaking, we say that the system exhibits replica symmetry if $R_{1,2} \approx \smallavg{R_{1,2}}$ with high probability (i.e., probability $\to 1$ as the system size $\to \infty$). Otherwise, we say that replica symmetry breaks. The breaking of replica symmetry is usually quite difficult to prove rigorously. RSB has been established rigorously only in mean-field systems, where every particle interacts with every other particle. The primary example of this is the Sherrington--Kirkpatrick (SK) model~\cite{sherringtonkirkpatrick75}, where the discovery of RSB led to the development of Parisi's broken replica method~\cite{mezardetal87}.  Rigorous proofs of RSB in the SK and other mean-field models came much later (see \cite{talagrand10, talagrand11} and references therein). 

For short-range models such as the RFIM and the Edwards--Anderson (EA) model~\cite{edwardsanderson75}, there is no proof of RSB as of now. Settling a longstanding debate~\cite{krzakalaetal10}, it was shown in \cite[Lemma 2.6]{chatterjee15} that replica symmetry does not break in the RFIM at any fixed temperature and nonzero field strength. The question of RSB in the EA model is still open, although some aspects of spin glass behavior have been established at zero temperature in the recent preprint~\cite{chatterjee23}, confirming some old conjectures from physics~\cite{fisherhuse86, braymoore87}.

The second basic property of spin glass models in Parisi's formulation is non-self-averaging (NSA). NSA is the property that the quenched law of $R_{1,2}$ (i.e., its law conditional on a realization of $\mu$) does not converge to a deterministic limit in probability as the system size goes to infinity. Rigorous proofs of NSA are now known for mean-field systems~\cite{talagrand10, talagrand11}, but there is no short-range model that has been rigorously proved to have the NSA property. In fact, there are mathematical arguments based on ergodic theory that seem to rule out NSA in translation-invariant short-range models in infinite volume~\cite{newmanstein96}, but there is a counter-argument that infinite volume systems do not truly represent finite volume behavior~\cite{parisi96}. The main result of \cite{chatterjee15} implies that the RFIM does not have the NSA property at non-critical field strengths, but that leaves open the possibility that NSA may hold at critical temperatures in the RFIM. Nothing is known about NSA in the EA model.

The third basic property of spin glasses is ultrametricity. This means, roughly speaking, that for any given $\ve > 0$, the probability of the event $R_{1,3}\ge \min\{R_{1,2}, R_{2,3}\} - \ve$ tends to $1$ as the system size goes to infinity. Ultrametricity implies that the Gibbs measure ``organizes the states like a tree'', a notion that has recently been made mathematically precise in \cite{chatterjeesloman21}. Ultrametricity also has the important consequence that it allows one to write down the joint distribution of arbitrarily many overlaps from the distribution of a single overlap, and thereby understand almost everything about the system. Ultrametricity has been rigorously proved in mean-field systems, most notably by \citet{panchenko13a} for a variant of the SK model, followed by extensions to other mean-field systems~\cite{auffingerchen16, jagannath17, contuccietal13, subag17, subag18}. As of today, there are no rigorous results about ultrametricity for systems with purely local interactions.

The fourth property --- the existence of many pure states --- means, very roughly, that the Gibbs measure behaves like a mixture of a large number of ergodic measures. Again, this is known rigorously only for mean-field models~\cite{panchenko13} and certain special models on lattices~\cite{newmanstein96a}. Incidentally, it is rather unclear how to define a pure state outside the setting of Markov random fields on finite-dimensional lattices~\cite{georgii11}. For certain kinds of mean-field spin glasses, a rigorous definition was given by \citet{panchenko13}. In the next section, we will give a general definition of the number of pure states that encompasses both mean-field and lattice models.% Constructing a physically natural finite-dimensional model which possesses many pure states remains an outstanding problem.

%More generally, constructing a short-range model of a disordered system that has the above features of mean-field spin glasses has long been one of the main unsolved questions in this area, first posed in the seminal monograph of \citet*{mezardetal87}. 

Proving that short-range models of disordered systems can have the above features of mean-field spin glasses has long been one of the main unsolved questions in this area, first posed in the seminal monograph of \citet*{mezardetal87}. As mentioned above, there is a negative result from \cite{chatterjee15}, where it was established that the RFIM does not have a phase where replica symmetry breaks. In this article, we show that in spite of this negative result, the first three features of a spin glass listed above --- RSB, NSA, and ultrametricity ---  can in fact arise in the RFIM, if instead of keeping the field strength $h$ in \eqref{original} fixed, we take it to zero like $|\Lambda|^{-\frac{1}{2}}$ as $|\Lambda|\to \infty$, and take $\beta$ bigger than the critical inverse temperature of the Ising model. Moreover, if the $J_i$'s are Gaussian, then we show that this is the only scaling of $h$ where this happens. However, the fourth property does not hold, because the system appears to be a mixture of two pure states instead of many. Another common (but perhaps not essential) feature of mean-field spin glasses, called the Ghirlanda--Guerra identities, also does not hold for this system.  %Lastly, we show that if one replace the ferromagnetic interactions in the RFIM with antiferromagnetic interactions, then in addition to the above properties, we also have that the magnetization tends zero, implying a certain kind of long-range disorder.

\section{Results}\label{resultsec}
Take any $d\ge 2$. For each $n$, let $B_n := \{-n, \ldots, n\}^d$, and let $E_n$ be the set of undirected nearest neighbor edges of $B_n$. Let $\Sigma_n := \{-1,1\}^{B_n}$ be the set of $\pm1$-valued spin configurations on $B_n$. Let $(J_i)_{i\in B_n}$ be a collection of i.i.d.~random variables with mean zero, variance one, and finite moment generating function in an open neighborhood of the origin. Let $h\in \R$ be a parameter. Define the Hamiltonian $H_n: \Sigma_n \to \rr$ as 
\begin{align}\label{hamil}
H_n(\sigma) := -\sum_{\{i,j\}\in E_n} \sigma_i \sigma_j  - \frac{h}{\sqrt{|B_n|}} \sum_{i\in B_n} J_i \sigma_i
\end{align}
This is the Hamiltonian for the Ising model on $B_n$ subjected to a random external field of strength $h J_i |B_n|^{-\frac{1}{2}}$ at site $i$ for each $i\in B_n$. That is, we have replaced the parameter $h$ in \eqref{original} by $h|B_n|^{-\frac{1}{2}}$. The Gibbs measure for this model at inverse temperature $\beta$ is the random probability measure on $\Sigma_n$ with probability mass function proportional to $e^{-\beta H_n(\sigma)}$ at each $\sigma \in \Sigma_n$. For a function $f:\Sigma_n\to \R$, let $\smallavg{f}$ denote its expected value with respect to the Gibbs measure. The ``quenched distribution'' of $f$ is the law of $f(\sigma)$ conditional on $(J_i)_{i\in B_n}$, where $\sigma$ is drawn from the Gibbs measure. %Lastly, let us define a quantity that appears several times below. Let
%Note that by the central limit theorem, $X_n$ converges in law to a Gaussian random variable with mean zero and variance $q\beta^2h^2$ as $n\to\infty$.

\subsection{Replica symmetry breaking and non-self-averaging}
Let $\sigma^1$ and $\sigma^2$ be drawn independently from the Gibbs measure defined by a single realization of the disorder $(J_i)_{i\in B_n}$. Recall from the previous section that the site overlap (or spin overlap) between $\sigma^1$ and $\sigma^2$ is defined as
\[
R_{1,2} := \frac{1}{|B_n|} \sum_{i\in B_n} \sigma_i^1 \sigma_i^2. 
\]
If we have a sequence of configurations $\sigma^1,\sigma^2,\ldots$ drawn independently from the Gibbs measure, then $R_{i,j}$ denotes the overlap between $\sigma^i$ and $\sigma^j$. The following theorem is the first main result of this paper.
%Our next result is that at sufficiently low temperatures, the overlap exhibits one step of replica symmetry breaking, meaning that with high probability, it is close to one of two possible values. %, but that cannot be narrowed down to one value. 
\begin{thm}[Replica symmetry breaking and non-self-averaging]\label{rsbthm}
Take any $d\ge 2$ and $n\ge 1$ and consider the model defined above on $B_n = \{-n,\ldots,n\}^d$ at inverse temperature $\beta>\beta_0$, where $\beta_0$ is the critical inverse temperature for the ordinary Ising model on $\Z^d$. Then there is a deterministic value $q>0$ depending only on $\beta$ and $d$, such that $\E\smallavg{(R_{1,2}^2 - q^2)^2} \to 0$ as $n\to \infty$. Moreover, if we define
\begin{align}\label{xndef}
X_n := \frac{\sqrt{q} \beta h}{\sqrt{|B_n|}} \sum_{i\in B_n} J_i,
\end{align}
then we have that %with $X_n$ defined as in \eqref{xndef}, we have that
\begin{align}\label{r12}
\lim_{n\to\infty} \E[(\smallavg{R_{1,2}} - q \tanh^2 X_n)^2] = 0.
\end{align}
Consequently, as $n\to\infty$, $\smallavg{R_{1,2}}$ converges in law to $q\tanh^2(\sqrt{q} \beta h Z)$, where $Z$ is a standard Gaussian random variable. 
\end{thm}
For the reader's convenience, let us briefly explain the significances of the two assertions of the above theorem. The first assertion, that $\E\smallavg{(R_{1,2}^2 - q^2)^2} \to 0$ as $n\to\infty$, shows that  when $n$ is large, the overlap $R_{1,2}$ is close to either $q$ or $-q$ with high probability. The second assertion shows that the quenched expectation of $R_{1,2}$ is a random variable that converges to a non-degenerate limiting distribution as $n\to\infty$. Jointly, this proves two things. First, it shows that $R_{1,2}$ does indeed behave like a random variable that is close to one of two values, and not just one value (because otherwise, $\smallavg{R_{1,2}}$ would be close to $q$ or $-q$). This is known as one step of replica symmetry breaking (1RSB). Second, it shows that the quenched distribution of the overlap is not self-averaging --- that is, it does not converge to a deterministic limiting distribution as $n\to \infty$. Equation \eqref{r12} shows that the mass near $q$ is approximately
\begin{align}\label{tanhform}
\frac{1}{2}(1+ \tanh^2 X_n),
\end{align}
and the mass near $-q$ is $1$ minus the above. An important thing to note is that $q$ depends only on $\beta$ and $d$, and not on $h$. Thus, $q$ is the limiting absolute value of the overlap in the ordinary Ising model --- that is, the case $h=0$. In particular, Theorem \ref{rsbthm} implies that for the Ising model, the quenched law of $R_{1,2}$ converges in probability to the uniform distribution on $\{-q,q\}$ as $n\to \infty$. The presence of $h$ only changes the masses near $q$ and $-q$.

Theorem \ref{rsbthm} shows that non-self-averaging can occur even in a system that only has local interactions. It is to be noted that the system under consideration here has no obvious representative in the infinite volume limit (because the field strength is tending to zero but with a non-trivial effect which cannot be captured by a model in infinite volume in any obvious way), thereby posing no contradiction to the results of \citet{newmanstein96} on the impossibility of NSA in translation-invariant infinite volume systems.

\subsection{Ultrametricity}
The next result says that the overlap satisfies the Parisi ultrametric property in the large $n$ limit, meaning that $R_{1,3}\ge \min\{R_{1,2},R_{2,3}\} - o(1)$ with probability $1-o(1)$ as $n\to \infty$.
\begin{thm}[Ultrametricity]\label{ultrathm}
Let $d$, $n$, $\beta_0$, $\beta$ and $q$ be as in Theorem \ref{rsbthm}. Then, as $n\to\infty$, the quenched distribution of $(R_{1,2}, R_{1,3}, R_{2,3})$ converges in law to a random limiting distribution with support 
\begin{align}\label{qqq}
\{(q,q,q), (-q,-q,q), (-q,q,-q), (q,-q,-q)\}.
\end{align}
Consequently, for any $\ve >0$, the quenched probability of the event $R_{1,3} \ge \min\{R_{1,2},R_{2,3}\}-\ve$ tends to $1$ in probability as $n\to\infty$.
%\[
%\lim_{n\to\infty} \E\smallavg{(\min\{R_{1,2},R_{1,3}\} - R_{2,3})^+} = 0,
%\]
%where $x^+$ denotes the positive part of a real number $x$.
 %Moreover, $\E\smallavg{(R_{1,2}-q)^2}$ and $\E\smallavg{(R_{1,2}+q)^2}$  do not tend to zero as $n\to\infty$. 
\end{thm}
Combined with Theorem \ref{rsbthm}, it is easy to deduce the approximate masses assigned by the law of $(R_{1,2}, R_{1,3}, R_{2,3})$ near the four points displayed in \eqref{qqq}. Let $a$ be the approximate mass near $(q,q,q)$, and let $b$ be the approximate mass near each of the other three points (which must be  equal, by symmetry). Then $a+3b \approx 1$, and $a+b \approx$ the probability of the event  $R_{1,2}\approx q$, which is given by the formula \eqref{tanhform}. Solving, we get
\[
a \approx \frac{1}{4}(1+ 3\tanh^2 X_n), \ \ \ b \approx \frac{1}{4}(1- \tanh^2 X_n).
\]
Just like Theorem \ref{rsbthm}, Theorem \ref{ultrathm} is valid even if $h=0$, that is, for the Ising model. It shows that at subcritical temperatures, the overlap in the Ising model has the ultrametricity property. 

\subsection{Behavior of the magnetization}
The magnetization of a configuration $\sigma$ is defined as
\[
m = m(\sigma) := \frac{1}{|B_n|}\sum_{i\in B_n } \sigma_i. 
\]
The following theorem identifies the limiting behavior of the magnetization of a configuration drawn from the Gibbs measure when $\beta$ is bigger than the critical inverse temperature of the Ising model. It also gives a relation between the magnetizations of two independently drawn configurations and their overlap.
\begin{thm}[Behavior of the magnetization]\label{magthm}
Let $d$, $n$, $\beta_0$, $\beta$ and $q$ be as in Theorem \ref{rsbthm}. The magnetization $m$ of a configuration $\sigma$ drawn from the model satisfies $\ee\smallavg{(m^2 - q)^2} \to 0$ as $n\to \infty$, and with $X_n$ defined as in \eqref{xndef}, we have
\begin{align}\label{mform}
\lim_{n\to\infty} \E[(\smallavg{m} - \sqrt{q} \tanh X_n )^2] = 0.
\end{align}
In particular, $\smallavg{m}$ converges in law to $\sqrt{q}\tanh(\sqrt{q} \beta h Z)$, where $Z$ is a standard Gaussian random variable. Moreover, for most values of $j\in B_n$, $\smallavg{\sigma_j}\approx \smallavg{m}$ with high probability, in the sense that
\begin{align}\label{mformgen}
\lim_{n\to\infty} \frac{1}{|B_n|} \sum_{j\in B_n} \E[(\smallavg{\sigma_j} - \smallavg{m})^2] =0. 
\end{align}
Lastly, if $m(\sigma^1)$ and $m(\sigma^2)$ are the magnetizations in two configurations $\sigma^1$ and $\sigma^2$ chosen independently from the same Gibbs measure, then $\E\smallavg{(R_{1,2} - m(\sigma^1)m(\sigma^2))^2} \to 0$ as $n\to\infty$.
\end{thm}
This theorem is the basis for proving the previously stated results about the overlap, because it says that the overlap between two configuration is approximately equal to the product of their magnetizations with high probability, and gives the approximate distribution of the magnetization, which is concentrated near $q$ or $-q$ with high probability. In addition to the previously stated results, it also gives the asymptotic quenched distribution  of any number of overlaps, because conditional on the disorder, $m(\sigma^1), m(\sigma^2),\ldots$ behave like i.i.d.~random variables taking values in $\{-\sqrt{q}, \sqrt{q}\}$ with a certain distribution, and $R_{i,j} \approx m(\sigma^i)m(\sigma^j)$ for each $i\ne j$.

\subsection{Two pure states}
As mentioned in the introduction, it is unclear how to rigorously define pure states outside the setting of Markov random fields on a lattice, where it is well-understood~\cite{georgii11}. We will now give a general definition of the number of pure states in a sequence of models, and show that according to this definition, our model has two pure states in the $n\to\infty$ limit. 

Let $\{N_n\}_{n\ge 1}$ be a sequence of positive integers tending to infinity, and let $(X_{n,i})_{n\ge 1, \, 1\le i\le N_n}$ 
be a triangular array of real-valued random variables. For each $n$, let $\pi_n$ be a uniform random permutation of $1,\ldots,N_n$, independent of the $X_{n,i}$'s. Let $Y_{n,i} := X_{n, \pi_n(i)}$. Let $Z = (Z_1,Z_2,\ldots)$ be a sequence of random variables such that for each $k$, $(Y_{n,1},\ldots, Y_{n,k})$ converges to $(Z_1,\ldots,Z_k)$ in distribution as $n\to\infty$. Then note that $Z$ is an infinite exchangeable sequence of random variables. By De Finetti's theorem~\cite[Theorem 1.1]{kallenberg05}, the law of $Z$ is a mixture of probability laws of i.i.d.~sequences, with a unique mixing measure~\cite[Proposition 1.4]{kallenberg05}. 
\begin{defn}\label{statedef1}
In the above setting, let $\mu$ be the mixing measure of the law of $Z$. Let $p$ be the size of the support of $\mu$, which may be a positive integer or infinity. Then, we will say that the law of $(X_{n,i})_{1\le i\le N_n}$ has $p$ pure states asymptotically as $n\to\infty$.
\end{defn}
For example, if the $X_{n,i}$'s are i.i.d., then so are the $Z_i$'s, and therefore $p=1$. On the other hand, suppose that $N_n = n$ and $X_{n,i} = Y + W_i$, $i=1,\ldots,n$, where $Y$ and $W_1,W_2,\ldots$ are i.i.d.~standard Gaussian random  variables. If $\pi_n$ is a uniform random permutation of $1,\ldots,n$, then for any $n$ and $k$, the law of $(X_{n,\pi_n(1)},\ldots, X_{n,\pi_n(k)})$ is the same as the law of $(Z_1,\ldots,Z_k)$, where $Z_i = Y+W_i$. Now, $Z_1,Z_2,\ldots$ is an infinite exchangeable sequence, which is conditionally i.i.d.~given $Y$. Since the support of $Y$ contains infinitely many points, we deduce that the law of $(X_{n,i})_{1\le i\le n}$ has infinitely many pure states as $n\to\infty$. 

In the setting of disordered systems, the law of $(X_{n,i})_{1\le i\le N_n}$ is itself random, and may not be converging to a deterministic limit in any reasonable sense as $n\to\infty$. Thus, we have to modify Definition \ref{statedef1} to accommodate this scenario. Let $Y_{n,i} = X_{n,\pi_n(i)}$ be defined as before. For each $k$, let $\nu_{n,k}$ be the law of $(Y_{n,1},\ldots,Y_{n,k})$, which is now a random probability measure. Let $\nu$ be a random probability measure taking value in the set of laws of infinite exchangeable sequences. Let $\nu_k$ be the (random) law of the first $k$ coordinates of a sequence with law $\nu$. 
\begin{defn}\label{statedef2}
Let $\nu$ be as above, and let $\mu$ be the (random) mixing measure of a random probability measure with law $\nu$.  Suppose that there is a deterministic $p\in \{1,2,\ldots\}\cup \{\infty\}$ such that with probability one, the support of $\mu$ has $p$ points. Also, suppose that for each $k$, the law of $\nu_{n,k}$ converges weakly to the law of $\nu_k$. Then, we will say that the (random) law of $(X_{n,i})_{1\le i\le N_n}$ has $p$ pure states asymptotically as $n\to\infty$. 
\end{defn}
The following result shows that under the above definition, our model has two pure states asymptotically as $n\to\infty$. This holds for any $h$, and in particular $h=0$, which is the case of the ordinary Ising model.
\begin{thm}\label{purethm}
Let $d$, $n$, $\beta_0$ and $\beta$ be as in Theorem \ref{rsbthm}. Then the random probability measure on $\Sigma_n$ defined by the model from Theorem \ref{rsbthm} has two pure states asymptotically as $n\to\infty$, as defined in Definition \ref{statedef2}. 
\end{thm}

\subsection{Failure of the Ghirlanda--Guerra identities}
The Ghirlanda--Guerra (GG) identities are a set of identities that are satisfied in the infinite volume limits of many mean-field spin glass models~\cite{ghirlandaguerra98}. A symmetric array of random variables $(S_{i,j})_{1\le i,j<\infty}$ is said to satisfy the GG identities if for any $k$, any bounded measurable function $f$  of $(S_{i,j})_{1\le i,j\le k}$, and any bounded measurable function $\psi:\R\to \R$, 
\begin{align}\label{ggid}
\E(f \psi(S_{1,k+1})) = \frac{1}{k}\E(f) \E(\psi(S_{1,2})) + \frac{1}{k}\sum_{i=2}^k \E(f \psi(S_{1,i})). 
\end{align}
These identities have been proved for the limiting joint law of overlaps for a variety of mean-field models of spin glasses. (Here, the ``joint law'' refers to the unconditional distribution, averaged over the disorder.) They form the basis of Panchenko's proof of ultrametricity in \cite{panchenko13a}, following a line of prior work connecting the GG identities with ultrametricity~\cite{aizenmancontucci98, arguinaizenman09, panchenko10, talagrand10a}. The following theorem shows that the GG identities are not valid for our model. This shows that while the GG identities are sufficient for ultrametricity of the overlap (as shown by Panchenko~\cite{panchenko13a}), they are not necessary.
\begin{thm}[Failure of the Ghirlanda--Guerra identities]\label{ggthm}
Let $d$, $n$, $\beta_0$ and $\beta$ be as in Theorem \ref{rsbthm}. Then the limiting joint distribution of the overlaps, as $n\to\infty$, does not satisfy the Ghirlanda--Guerra identities. 
\end{thm}

\subsection{Failure of spin glass behavior at other field strengths}
One may wonder if taking the field strength to be  proportional to $|B_n|^{-\frac{1}{2}}$ is the only way to get  replica symmetry breaking and non-self-averaging in the large $n$ limit. Our next result shows that this is indeed the case for Gaussian disorder (and it is reasonable to conjecture that the same holds for any i.i.d.~disorder). Replica symmetry does not break if the parameter $h$ is allowed to go to $\pm\infty$ as $n\to\infty$, and the non-self-averaging of the quenched law of the overlap breaks down if $h$ is allowed to go to zero as $n\to\infty$. %Thus, fixing $h$ is the only way to get the glassy behaviors given by the above theorems.
\begin{thm}[Failure of spin glass behavior at other field strengths]\label{negthm}
Suppose that  the parameter $h$ in the Hamiltonian $H_n$ is allowed to vary with $n$. If $h\to0$ as $n\to\infty$, then the distance between the quenched law of $R_{1,2}$ under our model and the law of $R_{1,2}$ under the Ising model on $B_n$ at the same temperature and free boundary condition tends to zero in probability as $n\to\infty$, for any metric that metrizes weak convergence of probability measures. In particular, non-self-averaging fails. On the other hand, if $|h|\to \infty$ as $n\to\infty$, and if the $J_i$'s are i.i.d.~standard Gaussian random variables, then $\E\smallavg{(R_{1,2}-\smallavg{R_{1,2}})^2} \to 0$, meaning that replica symmetry does not break. These conclusions hold at any temperature.  %then the set of limits points of the quenched law of $R_{1,2}$ as $n\to\infty$ is a deterministic set. %Moreover, if $\beta$ is large enough, then this set consists solely of the uniform distribution on $\{-q,q\}$, where $q$ is as in Theorem \ref{rsbthm}.
\end{thm}
The second assertion of the above theorem extends \cite[Lemma 2.6]{chatterjee15} by showing that replica symmetry holds not only when the parameter $h$ in the standard form \eqref{original} of the RFIM Hamiltonian is fixed and nonzero, but is even allowed to go to zero slower than $|\Lambda|^{-\frac{1}{2}}$ (for $\Lambda = B_n$).

\subsection{The antiferromagnetic RFIM}
For the sake of completeness, let us also consider the random field antiferromagnetic Ising model on $B_n$ under free boundary condition. This is the model where the minus in front of the first term on the right side in  \eqref{hamil} is replaced by a plus. That is, the Hamiltonian is
\begin{align}\label{antihamil}
H_n(\sigma) := \sum_{\{i,j\}\in E_n} \sigma_i \sigma_j  - \frac{h}{\sqrt{|B_n|}} \sum_{i\in B_n} J_i \sigma_i. 
\end{align}
All of the results for the ferromagnetic model continue to hold for the antiferromagnetic version, except one --- the magnetization tends to zero instead of converging in law to a non-degenerate distribution.
\begin{thm}[Results for the antiferromagnetic RFIM]\label{antithm}
Theorems \ref{rsbthm}, \ref{ultrathm} and \ref{negthm} remain valid for the antiferromagnetic model, with $J_i$ replaced by $(-1)^{|i|_1}J_i$ in the \eqref{xndef}, where $|i|_1$ is the $\ell^1$ norm of $i$. The magnetization, however, satisfies $\E\smallavg{m^2} \to 0$ as $n\to \infty$.
\end{thm}

\subsection{Uniformity of correlations in the ordinary Ising model}
In addition to the above results, our analysis also reveals the following ``uniformity of correlations'' for the ordinary  Ising model on $B_n$ under free boundary condition and subcritical temperatures. Namely, $\smallavg{\sigma_i\sigma_j}\approx q$ for most $i,j\in B_n$. More generally, for any even $l$ and most $i_1,\ldots, i_l\in B_n$, $\smallavg{\sigma_{i_1}\cdots \sigma_{i_l}} \approx q^{\frac{l}{2}}$. This result is the foundation for most of the other results in this paper.  Note that the correlation is zero if $l$ is odd due to the invariance of the model under the transform $\sigma \to -\sigma$.
\begin{thm}[Uniformity of correlations in the Ising model]\label{diffthmmain}
Let $d$, $n$, $\beta_0$, $\beta$ and $q$ be as in Theorem \ref{rsbthm}. Consider the ferromagnetic Ising model on $B_n$ at inverse temperature $\beta$ and free boundary condition (i.e., the model with Hamiltonian given in \eqref{hamil} but with $h=0$). Then for any even positive integer $l$,
\[
\lim_{n\to\infty} \frac{1}{|B_n|^l} \sum_{i_1,\ldots,i_l\in B_n}|\smallavg{\sigma_{i_1}\cdots\sigma_{i_l}} - q^{\frac{l}{2}}| = 0.
\]
%where $\smallavg{\cdot}$ denotes averaging with respect to the Ising model on $B_n$ at inverse temperature $\beta$ and free boundary condition. Then $\delta_n \to 0$ as $n\to \infty$.
\end{thm}
Uniformity of correlations in infinite volume is a simple consequence of a result of \citet{bodineau06} (see also \cite{raoufi20}), which says that the infinite volume Gibbs measure for the Ising model under free boundary condition is the average of the infinite volume measures under plus and minus boundary conditions. %We give this short proof in Subsection \ref{infvolsec}. 

The finite volume result stated above does not follow easily from the infinite volume result, even though we know that correlations decay exponentially under plus and minus boundary conditions \cite{duminilcopinetal20}. This is because Bodineau's theorem does not imply that the finite volume Gibbs measure under free  boundary is approximately the average of the finite volume measures under plus and minus boundary conditions. The proof presented in this draft is due to Hugo Duminil-Copin (private communication). It uses the random cluster representation of the Ising model and Pisztora's renormalization scheme~\cite{pisztora96}. A different proof was given in the original draft of this paper, which had the disadvantage of not covering the full supercritical regime and has therefore been omitted.%While the second proof is shorter, I have decided to retain the original proof because of its self-contained nature and the new ideas (e.g., the duality relation) that may be useful for other purposes.

%We give two proofs of Theorem \ref{diffthmmain}. The first proof is a bare-hands argument based on a generalization of the Kramers--Wannier duality for the Ising model to arbitrary dimensions and a multi-scale argument via this duality, starting from the infinite volume result as a launchpad. The second proof

This completes the statements of the results. The rest of the paper is devoted to proofs.

\subsection*{Acknowledgements} 
I thank Louis-Pierre Arguin, Andrew Chen, Persi Diaconis, Hugo Duminil-Copin, Zhihan Li and Gourab Ray for many helpful comments and references. In particular, I thank Hugo for sketching the alternative proof of Theorem \ref{diffthmmain} and Gourab for helping expand the sketch to a complete argument, and the referee for explaining why the proof works for all subcritical temperatures. This work was partially supported by NSF grants DMS-2113242 and DMS-2153654.\\
\\
{\bf Data availability statement:}  Data sharing not applicable to this article as no datasets were generated or analysed during the current study.

%\subsection{Motivation and related literature}

\section{Proof of Theorem \ref{diffthmmain}}
We now present the proof of Theorem \ref{diffthmmain} due to Hugo Duminil-Copin (private communication), which uses coupling with the FK-Ising (random cluster) model.
\subsection{The FK-Ising model}
Recall that the FK-Ising model on $B_n$ is defined as follows~\cite{grimmett06}.  Let $E_n$ be the set of edges of $B_n$, as before, and let $\Omega_n := \{0,1\}^{E_n}$. Each element $\omega\in \Omega_n$ defines a graph on $B_n$, with (open) edges corresponding to those $e\in  E_n$ for which $\omega_e = 1$. Edges of $B_n$ that are not in this graph are said to be ``closed''. Let $E(\omega)$ denote the number of open edges and $k(\omega)$ denote the number of connected components of this graph. The FK-Ising model with parameter $p$, under free boundary condition, assigns a probability proportional to 
\begin{align}\label{fkform}
p^{E(\omega)} (1-p)^{|E_n|-E(\omega)} 2^{k(\omega)}
\end{align}
at each $\omega \in \Omega_n$. A different kind of boundary condition, called the ``wired boundary condition'', has an identical form of the probability mass function but with a different definition of $k(\omega)$. Under the wired boundary condition, all the boundary vertices of $B_n$ are assumed to be connected to each other, and so all connected components that touch the boundary are merged into a single component. Fixing $p$, we will denote probabilities computed under the free and wired boundary conditions by $P_n^0$ and $P_n^1$, respectively. 

It is known that the infinite volume limits of these measures exist and are equal if $p$ is not equal to its critical value (which corresponds to the critical $\beta$ in the Ising model if we reparametrize $p=1-e^{-2\beta}$)~(by \cite[Theorem 2.1]{bodineau06} and \cite[Theorem 5.3(b)]{grimmett95}); that is, for any event $A$ determined by finitely many edges, the limits $\lim_{n\to\infty} P_n^0(A)$ and $ \lim_{n\to\infty} P_n^1(A)$ exist and are equal. We will denote this limit by $P(A)$. 

%\cite[Theorem 5.33]{grimmett06}

Following standard convention, we will denote by $x\leftrightarrow y $ the event that two vertices $x$ and $y$ are connected by a path of open edges. Similarly, $x\leftrightarrow \partial B_n$ will denote the event that $x$ is connected by a path to the boundary of $B_n$, and $x\leftrightarrow \infty$ will denote the event that $x$ belongs to an infinite open cluster. It is known that when $p$ is greater than the critical value, the infinite volume FK-Ising model has a unique infinite open cluster with probability one~\cite[Theorem 2]{burtonkeane89} (see also \cite[Theorem 1.10]{duminil-copin17}). In the following, we will assume throughout that $p$ is greater than the critical value. 

%\cite[Theorem 5.99]{grimmett06}

Lastly, define
\begin{align}\label{qdefinition}
q := \lim_{n\to\infty} (P(0\leftrightarrow \partial B_n))^2,
\end{align}
where the existence of the limit follows from monotonicity of the probability as a function of $n$. We will hold this $q$ fixed throughout the remaining discussion. Note that
\begin{align}\label{newq}
P(0\leftrightarrow \infty) &= P(0 \leftrightarrow \partial B_n \text{ for all } n)\notag \\
&= \lim_{n\to\infty} P(0\leftrightarrow \partial B_n) = \sqrt{q}. 
\end{align}
The numbers $p$ and $q$ will remain fixed throughout the remainder of this section, unless otherwise mentioned. 

\subsection{Uniformity of connectivities in infinite volume}
The identity \eqref{newq} leads to the following lemma, which shows that $P(0\leftrightarrow x) \approx q$ whenever $|x|$ is large. 
\begin{lmm}\label{altlmm1}
For any $x$, $P(0\leftrightarrow x) \ge q$, and given any $\ve >0$, there exists $C$ depending on $\ve$ such that whenever $|x|_\infty>C$ (where $|x|_\infty$ denotes the $\ell^\infty$ norm of $x$), we have $P(0\leftrightarrow x) \le q + \ve$.
\end{lmm}
\begin{proof}
Take any $x$. By the uniqueness of the infinite open cluster, the event $0\leftrightarrow x$ is implied by the events that $0\leftrightarrow \infty$ and $x\leftrightarrow \infty$ (with probability one). By the FKG property and the identity \eqref{newq}, this implies that
\begin{align*}
P(0\leftrightarrow x) &\ge P(0\leftrightarrow \infty, \, x\leftrightarrow\infty)\\
&\ge  P(0\leftrightarrow \infty)P(x\leftrightarrow\infty) = q.
\end{align*}
This completes the proof of the lower bound.  Next, for each $n$, let $B_n(x)$ denote the cube $B_n$ shifted by $x$, that is, the set $x+B_n$. Let $\partial B_n(x)$ denote the boundary of $B_n(x)$. Take any $x\ne 0$ and  $k< \frac{1}{2}|x|_\infty-1$. Then the cubes $\partial B_k$ and $\partial B_k(x)$ are disjoint. Moreover, there is a finite set $S$ of edges in $\Z^d$ that are not edges of $B_k$ or $B_k(x)$, such that 
\begin{itemize}
\item  if the edges in $S$ are all open, then all vertices of $\partial B_k$ and $\partial B_k(x)$ are in the same connected component, and 
\item every edge that is incident to a vertex in $\partial B_k \cup \partial B_k(x)$ but is not an edge of $B_k$ or $B_k(x)$, is a member of $S$.
\end{itemize}
Let $F$ denote the event that all edges in $S$ are open. Conditional on $F$, the configurations of open edges in $B_k$ and $B_k(x)$ are independent, and follow the random cluster models on these cubes with wired boundary condition. Take any $l<k$, and let $E$ be the event $\{0\leftrightarrow \partial B_l\}\cap\{ x \leftrightarrow \partial B_l(x)\}$. Since $E$ and $F$ are increasing events and $P(F) >0$, the FKG property implies that $P(E|F) \ge P(E)$. Consequently, 
\begin{align*}
P(0\leftrightarrow x) &\le P(E) \le P(E|F)\\
&= P(0\leftrightarrow \partial B_l, \, x \leftrightarrow \partial B_l(x) | F) \\
&= P_k^1(0\leftrightarrow \partial B_l) P_k^1(x\leftrightarrow \partial B_l(x))\\
&= (P_k^1(0\leftrightarrow \partial B_l))^2. 
\end{align*}
Take any $\ve >0$. For fixed $l$, if $k$ is large enough, then 
\[
(P_k^1(0\leftrightarrow \partial B_l))^2 \le( P(0\leftrightarrow \partial B_l))^2 + \frac{\ve}{2}.
\]
But if $l$ is large enough, then $(P(0\leftrightarrow \partial B_l))^2 \le q + \frac{\ve}{2}$. Thus, if $|x|$ is large enough, then we can choose $l$ and $k$ so that both inequalities are satisfied. This proves the claimed upper bound.
\end{proof}

\subsection{Uniformity of two-point correlations in finite volume}
Our next goal, roughly speaking, is to show that the conclusion of Lemma \ref{altlmm1} holds even if we consider the model restricted to a cube, as long as $0$ and $x$ are not too close to the boundary of the cube. The following lemma provides the upper bound.
\begin{lmm}\label{fin1}
Given any $\ve >0$ and $n$, there is some $k>0$ depending only on $d$ and $\ve$ (and not on $n$), such that whenever $x,y\in B_n$ and $|x-y|_\infty>k$, we have $P^0_n(x\leftrightarrow y) \le q + \ve$.
\end{lmm}
\begin{proof}
It is a simple consequence of the FKG property that $P_n^0(x\leftrightarrow y)$ is an increasing function of $n$. As a result, we have 
\[
P_n^0(x\leftrightarrow y) \le P(x\leftrightarrow y)
\]
for any $x,y\in B_n$. But by Lemma \ref{altlmm1} and the translation-invariance of the infinite volume measure, there is some $k$ depending only on $d$ and $\ve$ such that $P(x\leftrightarrow y)\le q+\ve$ whenever $|x-y|_\infty>k$. This completes the proof of the lemma.
\end{proof}

The lower bound in a finite cube is more complicated. It requires a version of the so-called ``Pisztora renormalization argument''~\cite{pisztora96}, due to \citet*{duminilcopinetal20}. First, recall the definition of the FK-Ising model on $B_n$ under arbitrary boundary condition. A general boundary condition $\xi$ refers to a partition of the set of boundary vertices of $B_n$, where we think of all vertices within the same member of the partition as being connected, when defining $k(\omega)$ in \eqref{fkform}. So, for example, $\xi$ consists of only singletons for the free boundary condition, and $\xi$ consists of only the full set $\partial B_n$ for the wired boundary condition. Let $P_k^\xi$ denote the model on $B_k$ under boundary condition $\xi$. For a given realization of the model on $B_n$, we say that a ``block'' $B_k(x)\subseteq B_n$ is ``good'' if $x\in k \Z^d$, and the following hold:
\begin{itemize}
\item There is an open cluster in $B_k(x)$ that touches all faces of $B_k(x)$.
\item Any open path in $B_k(x)$ of length $k$ is contained in this cluster.
\end{itemize}
We will frequently refer to the above open cluster as the ``giant open cluster'' of $B_k(x)$. Two blocks $B_k(x)$ and $B_k(y)$ are said to be neighbors if $x$ and $y$ are neighbors in $k\Z^d$. Note that two neighboring blocks have a substantial overlap. In particular, if two neighboring blocks are both good (in a realization of the model), then the conditions imply that the large clusters in the two blocks also intersect. In this situation, we say that the two blocks are ``connected''.

The result of \citet{duminilcopinetal20} that we are going to use is that for any $k$, any boundary condition $\xi$ on $B_{2k}$ and any $p$ greater than the critical value, 
\begin{align}\label{pisztora}
P_{2k}^\xi(B_k \text{ is good}) \ge 1- e^{-ck},
\end{align}
where $c$ depends only on $p$ and $d$. This was proved for $d\ge 3$ in \cite[Equation (3.1)]{duminilcopinetal20} (see also \cite[Theorem 2.1]{bodineau05}). For $d=2$, it follows from the RSW estimate in \citet*{duminil-copinetal11}. A consequence of this is the following lemma.
\begin{lmm}\label{oldlmm}
Given any $\ve >0$, the following is true for all large enough even $k$ (with the threshold depending only on $d$ and $\ve$). Suppose that $B_k(x)$ and $B_k(y)$ are both contained in $B_n$, and are disjoint. Then, under $P_n^0$,  the probability that any open path of size $\frac{k}{2}$ in $B_k(x)$ is connected to any open path of size $\frac{k}{2}$ in $B_k(y)$ is at least $1-\ve$.
\end{lmm}
\begin{proof}
Let $P$ be an open path of length $\frac{k}{2}$ in $B_k(x)$. Consider blocks of the form $B_l(z)$, $z\in l\Z^d$, where $l=\frac{k}{2}$. Let $a$ be the starting point of $P$. By the nature of the blocks, there is at least one block $D$ such that $a\in D$ and the $\ell^1$ distance of $a$ from $\partial D$ is at least $\frac{ld}{2}$. (For example, if $a= (a_1,\ldots,a_d)$, then we can choose $D= B_l(z)$ where $z_i$ is the integer multiple of $l$ that is closest to $a_i$, so that $|(z_i \pm l)-a_i|\ge \frac{l}{2}$.) Then the part of $P$ starting from $a$ and continuing until the first time $P$ hits $\partial D$, has length at least $\frac{ld}{2}\ge l$ (because $d\ge 2$). Thus, if $D$ is a good block, then this part of $P$ lies within the giant open cluster of $P$ as in the definition of good block above.

If $k$ is chosen large enough (depending on $d$ and $\ve$), then \eqref{pisztora} shows that with probability at least $1-\frac{\ve}{2}$, all blocks intersecting $B_k(x)$ or $B_k(y)$ are good. On the other hand, as argued in the proof of \cite[Lemma 3.1]{duminilcopinetal20} with the help of the main result of \cite{liggettetal97}, the collection of good blocks forms a finitely dependent percolation process on $l\Z^d\cap B_n$, which dominates an i.i.d.~percolation process with parameter $q$, where $q$ can be made as close to $1$ as we want by choosing $k$ large enough. Consequently, by choosing $k$ large enough we can guarantee that with probability at least $1-\frac{\ve}{2}$, the giant open cluster of any good block intersecting $B_k(x)$ is connected to the giant open cluster of any good block intersecting $B_k(y)$. Combining this with our previous deductions, we get that with probability at least $1-\ve$, any open path of size $\frac{k}{2}$ in $B_k(x)$ is connected to any open path of size $\frac{k}{2}$ in $B_k(y)$.
\end{proof}
We are now ready to prove the lower bound.
\begin{lmm}\label{fin2}
Given any $n$ and $\ve >0$, there exist positive integers $k,l$ depending only on $d$ and $\ve$ (and not on $n$) such that whenever $x,y\in B_n$, $|x-y|_\infty\ge 2k$, and $x,y$ are at an $\ell^\infty$ distance at least $l$ from $\partial B_n$, we have $P_n^0(x\leftrightarrow y) \ge q-\ve$.
\end{lmm}
\begin{proof}
Take any $n$, $k$, $l$, $x$ and $y$ as in the statement of the theorem, where $k$ and $l$ will be chosen later. We will choose $k$ to be even and $k<l<n$. Let $E$ be the event that there is an open cluster $C_x$ in $B_k(x)\setminus B_{\frac{k}{2}}(x)$ connecting $\partial B_k(x)$ and $\partial B_{\frac{k}{2}}(x)$, and an open cluster $C_y$ in $B_k(y)\setminus B_{\frac{k}{2}}(y)$ connecting  $\partial B_k(y)$ and $\partial B_{\frac{k}{2}}(y)$, such that $C_x\not \leftrightarrow C_y$ in $B_n$. Note that if $x\leftrightarrow \partial B_k(x)$, $y\leftrightarrow \partial B_k(y)$, and $E$ fails to happen, then the open clusters connecting $x$ to $\partial B_k(x)$ and $y$ to $\partial B_k(y)$ must be connected in $B_n$, and therefore $x\leftrightarrow y$. Thus,  by the FKG inequality,
\begin{align}\label{pninequality}
P_n^0(x\leftrightarrow y) &\ge P_n^0(x\leftrightarrow \partial B_k(x), \, y\leftrightarrow \partial B_k(y)) - P_n^0(E)\notag \\
&\ge P_n^0(x\leftrightarrow \partial B_k(x))P_n^0( y\leftrightarrow \partial B_k(y)) - P_n^0(E).
\end{align}
Let $Q_x$ denote the  FK-Ising model on $B_l(x)$ under free boundary condition. Since $k<l<n$ and $B_l(x) \subseteq B_n$, the FKG property implies that 
\[
P_n^0(x\leftrightarrow \partial B_k(x)) \ge Q_x(x\leftrightarrow \partial B_k(x)) = P_l^0(0\leftrightarrow \partial B_k).
\]
Similarly, 
\[
P_n^0(y\leftrightarrow \partial B_k(y))  \ge P_l^0(0\leftrightarrow \partial B_k). 
\]
By the definition of $q$, we can choose $k$ large enough (depending on $d$ and $\ve$) such that 
\[
P(0\leftrightarrow \partial B_k) \ge \sqrt{q}-\frac{\ve}{8}.
\]
Having chosen $k$ like this, we can then use the definition of the infinite volume measure to find $l$ large enough (depending on $d$, $\ve$ and $k$) such that
\[
P_l^0(0\leftrightarrow \partial B_k) \ge P(0\leftrightarrow \partial B_k) -\frac{\ve}{8}.
\]
Thus, with such choices of $k$ and $l$, we get that $P_n^0(x\leftrightarrow \partial B_k(x))$ and $P_n^0(y\leftrightarrow \partial B_k(y))$ are both bounded below by $\sqrt{q} - \frac{\ve}{4}$. Plugging this into \eqref{pninequality}, we get
\[
P_n^0(x\leftrightarrow y) \ge q - \frac{\ve}{2} - P_n^0(E). 
\]
Finally, with a large enough choice of $k$ (depending on $d$ and $\ve$), Lemma \ref{oldlmm} implies that $P_n^0(E)<\frac{\ve}{2}$, which completes the proof.  
\end{proof}

Using Lemmas \ref{fin1} and \ref{fin2}, and a standard coupling of the Ising and FK-Ising models, we are now ready to state and prove the following theorem, which is the main result of this subsection.
\begin{thm}\label{diffthm}
Let $\beta_0$ be the critical temperature of the ordinary Ising model in dimension $d$. Take any $\beta> \beta_0$ and let $p:=1-e^{-2\beta}$, so that $p$ is a supercritical probability for the FK-Ising model. Let $q$ be defined as in equation~\eqref{qdefinition}. Take any $\ve \in (0,1)$. For each $n$, let 
\[
\delta_n = \delta_n(\beta,d,\ve):= \max\{|\smallavg{\sigma_i\sigma_j} - q|: i,j\in B_{\lfloor (1-\ve)n\rfloor}, \, |i-j|_1\ge \ve n\},
\]
where $\smallavg{\cdot}$ denotes averaging with respect to the Ising model on $B_n$ at inverse temperature $\beta$ and free boundary condition. Then $\delta_n \to 0$ as $n\to \infty$.
\end{thm}

\begin{proof}%[Proof of Theorem \ref{diffthm}]
Take any $n$. Let $P_n^0(\cdot)$ denote probability computed under the FK-Ising model with parameter $p$ on $B_n$ under free boundary condition. It is a standard fact~\cite[Theorem 1.16]{grimmett06} that for any $i,j\in B_n$, 
\begin{align}\label{twoptid}
\smallavg{\sigma_i\sigma_j} = P_n^0(i\leftrightarrow j).
\end{align}
Using this identity and Lemmas \ref{fin1} and \ref{fin2},  it is now straightforward to prove Theorem \ref{diffthm}. 
\end{proof}

%\section{Preliminaries}

\subsection{Uniformity of four-point correlations in finite volume}
In this subsection, we show that for most quadruples of vertices $i,j,k,l\in B_n$, $\smallavg{\sigma_i\sigma_j\sigma_k\sigma_l} \approx q^2$ if $n$ is large, where the expectation is with respect to the Ising model with free boundary condition on $B_n$ at a supercritical inverse temperature $\beta$, and $q$ is as in \eqref{qdefinition}. First, we need the following analogue of Lemma \ref{fin1} for four-point connectivities. Let $p$ be as in Theorem \ref{diffthm}.
\begin{lmm}\label{four1}
Given any $\ve >0$ and $n$, there is some $k>0$ depending only on $d$ and $\ve$ (and not on $n$), such that if $x,y,w,z\in B_n$ are such that all interpoint $\ell^\infty$ distances are greater than $2k$, and all four points are at least at an $\ell^\infty$ distance $k$ from the boundary, then we have 
\begin{align*}
P_n^0(x\leftrightarrow y \leftrightarrow w \leftrightarrow z) \le q^2 + \ve, \ \ \ P_n^0(x\leftrightarrow y  \not\leftrightarrow w \leftrightarrow z)\le \ve.
\end{align*}
\end{lmm}
\begin{proof}
For the proof of the first inequality, we proceed as in the proof of Lemma \ref{altlmm1}. Since the interpoint distances are all greater than $2k$, the cubes $B_k(x)$, $B_k(y)$, $B_k(w)$ and $B_k(z)$ are disjoint. Let $F$ be the event that all edges of $B_n$ that do not belong to these cubes are open. Then by the FKG property, we have that for any $l<k$, 
\begin{align*}
P_n^0(x\leftrightarrow y \leftrightarrow w \leftrightarrow z) &\le P_n^0(x\leftrightarrow \partial B_l(x), \, y \leftrightarrow \partial B_l(y), \, w \leftrightarrow \partial B_l(w), \, z\leftrightarrow \partial B_l(z))\\
&\le P_n^0(x\leftrightarrow \partial B_l(x), \, y \leftrightarrow \partial B_l(y), \, w \leftrightarrow \partial B_l(w), \, z\leftrightarrow \partial B_l(z) | F).
\end{align*}
But, given $F$, the configurations inside the cubes $B_k(x)$, $B_k(y)$, $B_k(w)$ and $B_k(z)$ are independent, and follow the FK-Ising models in these cubes with wired boundary condition. Thus, we get
\begin{align*}
P_n^0(x\leftrightarrow y \leftrightarrow w \leftrightarrow z) &\le (P_k^1(0\leftrightarrow \partial B_l))^4. 
\end{align*}
By the equality of the infinite volume measures under free and wired boundary conditions, we have that for $l$ fixed and $k$ sufficiently large (depending on $l$, $d$ and $\ve$), 
\[
(P_k^1(0\leftrightarrow \partial B_l))^4 \le (P(0\leftrightarrow \partial B_l))^4 +\frac{\ve}{2}.
\]
By the definition of $q$, a large enough value of $l$ ensures that
\[
(P(0\leftrightarrow \partial B_l))^4 \le q^2 + \frac{\ve}{2}. 
\]
Combining the last three displays proves the first claim of the lemma. 

For the second claim, note that the event $x\leftrightarrow y  \not\leftrightarrow w \leftrightarrow z$ implies that $x\leftrightarrow \partial B_k(x)$, $y\leftrightarrow \partial B_k(y)$, $w\leftrightarrow \partial B_k(w)$, and $z\leftrightarrow \partial B_k(z)$, but the open cluster joining $y$ to $\partial B_k(y)$ is not connected to the open cluster joining $w$ to $\partial B_k (w)$. By Lemma \ref{oldlmm}, the probability of this event can be made as small as we like by choosing $k$ large enough. 
\end{proof}
The next lemma is the analogue of Lemma \ref{fin2} for four-point connectivities.
\begin{lmm}\label{four2}
Given any $\ve >0$ and $n$, there exist $k$ and $l$ depending only on $d$ and $\ve$ (and not on $n$), such that if $x,y,w,z\in B_n$ are such that all interpoint $\ell^\infty$ distances are greater than $2k$, and all four points are at least at an $\ell^\infty$ distance $l$ from the boundary, then we have 
\begin{align*}
P_n^0(x\leftrightarrow y \leftrightarrow w \leftrightarrow z) \ge q^2 - \ve.
\end{align*}
\end{lmm}
\begin{proof}
Let $C_x$ and $C_y$ be as in the proof of Lemma \ref{fin2}, and define $C_w$ and $C_z$ analogously. Let $E$ be the event at least two of these clusters are not connected to each other. Then, as in the proof of Lemma \ref{fin2}, we can use Lemma \ref{oldlmm} to conclude that $P_n^0(E)<\frac{\ve}{2}$ if $k$ is chosen large enough. Also, as in the proof of Lemma \ref{fin2}, we deduce that
\begin{align*}
&P_n^0(x\leftrightarrow y \leftrightarrow w \leftrightarrow z) \\
&\ge P_n^0(x\leftrightarrow \partial B_k(x),\, y \leftrightarrow \partial B_k(y), \, w \leftrightarrow \partial B_k(w), \, z\leftrightarrow \partial B_k(z)) - P_n^0(E).
\end{align*}
The proof is now completed by applying the FKG inequality to replace the first probability on the right by the product of the probabilities of the four events, and then proceeding as in the proof of Lemma \ref{fin2} to show that these probabilities are all bounded below by $\sqrt{q}-\frac{\ve}{100}$ if $k$ and $l$ are large enough.
\end{proof}
Finally, the following lemma gives the analogue of equation \eqref{twoptid} for four-point correlatons.
\begin{lmm}\label{fourrep}
Take any $\beta >0$ and let $p:=1-e^{-2\beta}$. Take any $n$. Let $\smallavg{\cdot}$ denote averaging with respect to the Ising model on $B_n$ at inverse temperature $\beta$ and free  boundary condition, and let $P_n^0(\cdot)$ denote probability computed under the FK-Ising model with parameter $p$ on $B_n$ under free boundary condition. Then for any distinct $i,j,k,l\in B_n$,
\begin{align*}
\smallavg{\sigma_i\sigma_j\sigma_k\sigma_l} &= P_n^0(\textup{Any open cluster contains even number of elements from $\{i,j,k,l\}$})\\
&= P_n^0(i\leftrightarrow j \leftrightarrow k \leftrightarrow l) + P_n^0( i\leftrightarrow j \not \leftrightarrow k \leftrightarrow l) \\
&\qquad + P_n^0(i\leftrightarrow k\not \leftrightarrow j\leftrightarrow l) + P_n^0(i\leftrightarrow l \not\leftrightarrow j \leftrightarrow k).
\end{align*}
\end{lmm}
\begin{proof}
It is a standard fact that a configuration from the Ising model on $B_n$ at inverse temperature $\beta$ and free boundary condition may be obtained as follows. First, generate a configuration from the FK-Ising model on $B_n$ with parameter $p$, under free boundary condition. Then, take the connected components of vertices in this configuration, and independently for each component, assign the same spin to all vertices, where the spin is chosen to be $1$ or $-1$ with equal probability. (For a proof, see \cite[Theorem 1.16]{grimmett06}.)

Now take any distinct $i,j,k,l\in B_n$. To compute $\smallavg{\sigma_i\sigma_j\sigma_k\sigma_l}$, we consider the above coupling and compute the conditional expectations given the FK-Ising configuration, which we denote by $\smallavg{\sigma_i\sigma_j\sigma_k\sigma_l}'$. The following are easy to see:
\begin{itemize}
\item If $i, j,k,l$ are all in the same cluster, them $\smallavg{\sigma_i\sigma_j\sigma_k\sigma_l}' = 1$.
\item If two of $i,j,k,l$ are in one cluster and the other two are in a different cluster, then $\smallavg{\sigma_i\sigma_j\sigma_k\sigma_l}' = 1$.
\item In all other cases, $\smallavg{\sigma_i\sigma_j\sigma_k\sigma_l}' = 0$.
\end{itemize}
Taking unconditional expectation gives the desired result.
\end{proof}
We now arrive at the main result of this subsection.
\begin{thm}\label{diffthm20}
Let $d$, $\beta_0$, $\beta$ and $q$ be as in Theorem \ref{diffthm}. Take any $\ve \in (0,1)$. For each $n$, let 
\begin{align*}
\gamma_n = \gamma_n(\beta,d,\ve)&:= \max\{|\smallavg{\sigma_i\sigma_j\sigma_k\sigma_l} - q^2|: i,j,k,l\in B_{\lfloor (1-\ve)n\rfloor}, \\
&\qquad \qquad \qquad \textup{all pairwise $\ell^1$ distances between $i,j,k,l$ are $\ge \ve n$}\},
\end{align*}
where $\smallavg{\cdot}$ denotes averaging with respect to the Ising model on $B_n$ at inverse temperature $\beta$ and free boundary condition. Then $\gamma_n \to 0$ as $n\to \infty$.
\end{thm}

\begin{proof}%[Alternative proof of Theorem \ref{diffthm20}]
It is easy to see that Theorem \ref{diffthm20} follows from the representation of the four-point correlation given in Lemma \ref{fourrep}, together with the upper and lower bounds given in Lemma \ref{four1} and Lemma~\ref{four2}. 
\end{proof}

\subsection{Concentration of the magnetization and the overlap in the Ising model}
To generalize the results for two-point and four-point correlations to $l$-point correlations for all even $l$, as well as for other purposes, we need the following theorem. 
\begin{thm}\label{diffthm2}
Let $d$, $\beta_0$, $\beta$ and $q$ be as in Theorem \ref{diffthm}. Let $m$ be the magnetization and $R_{1,2}$ be the overlap between two independent replicas in the ferromagnetic Ising model on $B_n$ under free boundary condition and inverse temperature $\beta$.  Then  $\smallavg{(m^2-q)^2}\to 0$ and $\smallavg{(R_{1,2}^2 - q^2)^2} \to 0$ as $n\to \infty$.
\end{thm}
\begin{proof}
Throughout this proof, $C, C_1, C_2,\ldots$ will denote constants that depend only on $d$, whose values may change from line to line. Fix some $n$ and some $\ve \in (0,1)$. Let $m := \lfloor (1-\ve) n\rfloor$. Let $\delta_n$ be as in Theorem \ref{diffthm} and $\gamma_n$ be as in Theorem \ref{diffthm20}. Let
\begin{align}\label{sdefin}
S := \{(i,j)\in B_n \times B_n: i,j\in B_m, \, |i-j|_1\ge \ve n\},
\end{align} 
and let 
\[
T := \{(i,j,k,l)\in B_m^4: \textup{all pairwise $\ell^1$ distances between $i,j,k,l$ are $\ge \ve n$}\}.
\]
Let $S^c := B_n^2 \setminus S$ and $T^c := B_n^4 \setminus T$. Now, if $(i,j)\in S^c$, then either at least one of $i$ and $j$ is in $B_n \setminus B_m$, or $|i-j|_1< \ve n$. From this observation, it follows that 
\begin{align}\label{scbound}
|S^c| \le C\ve n^{2d} + C\ve^d n^{2d},
\end{align}
where $C$ depends only on $d$. Since
\begin{align*}
|\smallavg{m^2} - q| &\le \frac{1}{|B_n|^2}\sum_{i,j\in B_n} |\smallavg{\sigma_i\sigma_j} - q| \le \frac{|S^c|}{|B_n|^2} + \frac{|S|\delta_n}{|B_n|^2},
\end{align*}
this shows that 
\begin{align}\label{mineq1}
\limsup_{n\to\infty}|\smallavg{m^2} - q|  \le C\ve.
\end{align}
On the other hand,
\begin{align*}
|\smallavg{m^4} - q^2| &\le \frac{1}{|B_n|^4} \sum_{i,j,k,l\in B_n} |\smallavg{\sigma_i\sigma_j\sigma_k\sigma_l} - q^2|\\
&\le \frac{|T^c|}{|B_n|^4} + \frac{|T|\gamma_n}{|B_n|^4},
\end{align*}
which shows that 
\begin{align}\label{mineq2}
\limsup_{n\to\infty}|\smallavg{m^4} - q^2|  \le C\ve.
\end{align}
Combining \eqref{mineq1} and \eqref{mineq2}, we get
\begin{align*}
\limsup_{n\to\infty} \smallavg{(m^2-q)^2} &= \limsup_{n\to\infty}( \smallavg{m^4} - 2\smallavg{m^2}q + q^2)\\
&=  \limsup_{n\to\infty}( \smallavg{m^4} - q^2 - 2(\smallavg{m^2}-q)q)\\
&\le C\ve. 
\end{align*}
Since $\ve$ is arbitrary, this completes the proof of the first assertion of the theorem. Next, note that
\begin{align*}
\smallavg{R_{1,2}^2} &= \bigavg{\biggl(\frac{1}{|B_n|}\sum_{i\in B_n} \sigma_i^1 \sigma_i^2\biggr)^2}\\
&= \frac{1}{|B_n|^2} \sum_{i,j\in B_n} \smallavg{\sigma_i^1 \sigma_i^2 \sigma_j^1 \sigma_j^2}\\
&= \frac{1}{|B_n|^2} \sum_{i,j\in B_n} \smallavg{\sigma_i\sigma_j}^2. 
\end{align*}
Proceeding as above, this shows that $\smallavg{R_{1,2}^2}\to q^2$ as $n\to\infty$. Similarly,
\begin{align*}
\smallavg{R_{1,2}^4} &= \frac{1}{|B_n|^4}\sum_{i,j,k,l} \smallavg{\sigma_i\sigma_j\sigma_k\sigma_l}^2,
\end{align*}
which can be used as above to show that $\smallavg{R_{1,2}^4}\to q^4$ as $n\to\infty$. Combining, we get that $\smallavg{(R_{1,2}^2-q^2)^2} \to 0$. 
\end{proof}

\begin{cor}\label{diffcor}
In the setting of Theorem \ref{diffthm2}, as $n\to\infty$, the law of $m$ tends to the probability measure that puts equal mass on $\pm \sqrt{q}$, and the law of $R_{1,2}$ tends to the probability measure that puts equal mass on $\pm q$. %The same conclusion holds for $R_{1,2}$ in the antiferromagnetic Ising model. 
\end{cor}
\begin{proof}
Simply combine Theorem \ref{diffthm2} with the observation that $\smallavg{R_{1,2}} = \smallavg{m}=0$ by the invariance of the model under the transform $\sigma \to -\sigma$. %For the antiferromagnetic model, observe that $R_{1,2}$ remains unchanged under the transformation that takes a configuration from the ferromagnetic model to a configuration from the antiferromagnetic model (as in the proof of Theorem \ref{magthm0}).  
\end{proof}

\subsection{Proof of Theorem \ref{diffthmmain}}\label{diffproofsec}
For $l=2$, the proof is contained in  the proof of Theorem \ref{diffthm2}. Take any even $l\ge 4$. Note that
\begin{align*}
\frac{1}{|B_n|^l}\sum_{i_1,\ldots,i_l\in B_n} \smallavg{\sigma_{i_1}\cdots \sigma_{i_l}} = \smallavg{m^l},
\ \ \ \frac{1}{|B_n|^l}\sum_{i_1,\ldots,i_l\in B_n} \smallavg{\sigma_{i_1}\cdots \sigma_{i_l}}^2 = \smallavg{R_{1,2}^l}.
\end{align*}
Thus, by the Cauchy--Schwarz inequality, %, $m(\sigma) \to \sqrt{q}$ in $L^2$. 
\begin{align}
&\frac{1}{|B_n|^l}\sum_{i_1,\ldots,i_l\in B_n} |\smallavg{\sigma_{i_1}\cdots \sigma_{i_l}} - q^{\frac{l}{2}}| \le \biggl[\frac{1}{|B_n|^l}\sum_{i_1,\ldots,i_l\in B_n} (\smallavg{\sigma_{i_1}\cdots \sigma_{i_l}} - q^{\frac{l}{2}})^2\biggr]^{\frac{1}{2}}\notag \\
&= \biggl[\frac{1}{|B_n|^l}\sum_{i_1,\ldots,i_l\in B_n} (\smallavg{\sigma_{i_1}\cdots \sigma_{i_l}}^2 - 2q^{\frac{l}{2}}\smallavg{\sigma_{i_1}\cdots \sigma_{i_l}} + q^l)\biggr]^{\frac{1}{2}}\notag \\
&= [\smallavg{R_{1,2}^l} - 2q^{\frac{l}{2}}\smallavg{m^l} + q^l]^{\frac{1}{2}}. \label{r12l}
\end{align}
Now, by the fact that $R_{1,2}$ and $q$ are both in $[0,1]$, and the inequality 
\[
|x^{\frac{l}{2}} - y^{\frac{l}{2}}|\le \frac{l}{2}|x-y|
\]
that holds for all $x,y\in [-1,1]$, we have 
\begin{align*}
|\smallavg{R_{1,2}^l} - q^l| &\le  \smallavg{|R_{1,2}^l - q^l|}\le \frac{l}{2}\smallavg{|R_{1,2}^2 - q^2|}. %\le \frac{l}{2} \smallavg{(R_{1,2}^2 - q^2)^2}.
\end{align*}
Thus, by Theorem \ref{diffthm2}, $\smallavg{R_{1,2}^l} \to q^l$ as $n\to \infty$. Similarly,
\[
|\smallavg{m^l} - q^{\frac{l}{2}}| \le \smallavg{|m^l - q^{\frac{l}{2}}|} \le\frac{l}{2} \smallavg{|m^2 - q|}, %\le\frac{l}{2} \smallavg{(m^2 - q)^2},
\]
and so, $\smallavg{m^l}\to q^{\frac{l}{2}}$ as $n\to \infty$. Using these in \eqref{r12l} completes the proof.

\section{Proofs of the main results}
In this section, we will complete the proofs of the results from Section \ref{resultsec} (except Theorem \ref{diffthmmain}, which has already been proved in the previous section). Throughout this section, we will let $\smallavg{\cdot}$ denote averaging with respect to the model on $B_n$ with Hamiltonian $H_n$ defined in \eqref{hamil}, at inverse temperature $\beta$. Diverging from the notation used in the previous section, we will use $\smallavg{\cdot}_0$ to denote averaging with respect to the Ising model on $B_n$ at inverse temperature $\beta$ and free boundary condition (because this model corresponds to the case $h=0$ of our model). The following lemma (which is just the central limit theorem for the moment generating function) will be used several times.
%$M:= \max_{i\in B_n}|a_i|$. Suppose that $M \le \frac{1}{2}\theta |B_n|^{\frac{1}{2}}$ for some $\theta\in [0,1]$ such that
\begin{lmm}\label{gaussianlmm}
Take any $a_i\in \R$, $i\in B_n$. Let $\theta>0$ be a constant such that $|a_i|\le \theta/2$ for all $i$, and  $\E(e^{\theta |J_0|})<\infty$. Then
\begin{align*}
\biggl|\E \biggl[\exp\biggl(\frac{1}{\sqrt{|B_n|}}\sum_{i\in B_n} a_i J_i\biggr)\biggr] - \exp\biggl(\frac{1}{2|B_n|}\sum_{i\in B_n} a_i^2 \biggr)\biggr| &\le \frac{C}{|B_n|^{\frac{3}{2}}} \sum_{i\in B_n} |a_i|^3, 
\end{align*}
where $C$ is a positive constant that depends only on the law of the $J_i$'s and the choice of $\theta$. 
\end{lmm}
\begin{proof}
%Thus, there exists $\theta>0$ such that $\E(e^{\theta |J_0|}) < \infty$. Take any such $\theta$, which further satisfies $M \le \frac{1}{2}\theta |B_n|^{1/2}$.
%Recall that the $J_i$'s have a finite moment generating function in an open neighborhood of the origin.  
Take any $\theta$ as in the statement of the theorem. We will let $C, C_1,C_2,\ldots$ denote any positive constants whose values depend only on $d$, on the law of the $J_i$'s and on the choice of $\theta$, and whose values may change from line to line. First, note that for any $k$,
\begin{align}\label{momentbd}
\E|J_0|^k &\le \frac{k!}{\theta^k} \E(e^{\theta|J_0|}) \le \frac{Ck!}{\theta^k}.
\end{align}
By the above inequality and the facts that $\E(J_0)=0$, $\E(J_0^2)=1$,  we get that for any $i$,
\begin{align}\label{exp1}
\E\bigg[\exp\biggl(\frac{a_i J_i}{\sqrt{|B_n|}}\biggr)\biggr]  &= 1 + \frac{a_i^2}{2|B_n|} + R_i, 
\end{align}
where
\[
R_i :=  \sum_{k=3}^\infty \E\biggl(\frac{a_i^k J_i^k}{k!|B_n|^{\frac{k}{2}}}\biggr).
\]
Note that by \eqref{momentbd} and the fact that $|a_i|\le \theta/2$ for all $i$, 
\begin{align}\label{exprem}
|R_i| &\le \sum_{k=3}^\infty\frac{|a_i|^k}{k! |B_n|^{\frac{k}{2}}}\E|J_i|^k \le \sum_{k=3}^\infty\frac{C_1|a_i|^3(\theta/2)^{k-3}}{|B_n|^{\frac{k}{2}}\theta^k}\le  \frac{C_2|a_i|^3 }{|B_n|^{\frac{3}{2}}}. 
\end{align}
Similarly, %since $M|B_n|^{-\frac{1}{2}}\le \frac{\theta}{2}\le \frac{1}{2}$, 
\begin{align}\label{exprem2}
\biggl|\exp\biggl(\frac{a_i^2}{2|B_n|}\biggr) - 1 - \frac{a_i^2}{2|B_n|}\biggr| &= \sum_{k=2}^\infty \frac{a_i^{2k}}{k! 2^k |B_n|^k} \le  \frac{Ca_i^4}{|B_n|^2}. 
\end{align}
Now, for any $N$, and any $x_1,\ldots, x_N, y_1,\ldots, y_N \in \R$, if we let $K$ be the maximum of $|x_i|$ and $|y_i|$ over all $i$, then
\begin{align}
\biggl|\prod_{i=1}^N x_i - \prod_{i=1}^N y_i\biggr| &\le \sum_{i=1}^{N} |x_1\cdots x_{i-1} y_{i}\cdots y_N - x_1\cdots x_i y_{i+1}\cdots y_N|\notag \\
&\le\sum_{i=1}^N K^{N-1}|x_i - y_i|. \label{momineq}
\end{align}
%Recall that $M = \max_{i\in B_n} |a_i|$. 
By \eqref{exp1}, \eqref{exprem}, and the inequalities $1+x\le e^x$ and $|a_i|\le \theta/2$, 
\begin{align*}
0\le \E \biggl[\exp\biggl(\frac{a_i J_i}{\sqrt{|B_n|}}\biggr)\biggr] \le e^{C/|B_n|}. %e^{C(M^2 + M^3)|B_n|^{-1}}. 
\end{align*}
Similarly, by \eqref{exprem2},
\[
0\le \exp\biggl(\frac{a_i^2}{2|B_n|}\biggr) \le e^{C/|B_n|}.
\]
Thus,  by  \eqref{exp1}, \eqref{exprem}, \eqref{exprem2} and \eqref{momineq}, 
\begin{align*}
&\biggl|\E \biggl[\exp\biggl(\frac{1}{\sqrt{|B_n|}}\sum_{i\in B_n} a_i J_i\biggr)\biggr] - \exp\biggl(\frac{1}{2|B_n|}\sum_{i\in B_n} a_i^2 \biggr)\biggr| \\
&= \biggl|\prod_{i\in B_n}\E \biggl[\exp\biggl(\frac{a_i J_i}{\sqrt{|B_n|}}\biggr)\biggr] - \prod_{i\in B_n}\exp\biggl(\frac{a_i^2}{2|B_n|} \biggr)\biggr|\\
&\le (e^{C/|B_n|})^{|B_n|} \sum_{i\in B_n} \biggl|\E \biggl[\exp\biggl(\frac{a_i J_i}{\sqrt{|B_n|}}\biggr)\biggr] - \exp\biggl(\frac{a_i^2}{2|B_n|} \biggr)\biggr|\\
&\le \frac{C}{|B_n|^{\frac{3}{2}}} \sum_{i\in B_n} |a_i|^3.
\end{align*}
This completes the proof of the lemma. 
\end{proof}
%\section{Magnetization and overlap in RFIM with weak external field}
\subsection{Proof of Theorem \ref{magthm}}
In this proof, $o(1)$ will denote any quantity, deterministic or random, whose absolute value can be bounded by a deterministic quantity depending only on $n$ (and the law of the $J_i$'s and our choices of $\beta$ and $d$) that tends to zero as $n\to\infty$. 
%The goal of this section is to prove Theorem \ref{magthm}. 
We begin with the derivation of the approximate formula for the quenched expectation of the magnetization. Let $X_n$ be defined as in \eqref{xndef},  
%\subsection{Approximate formula for expected magnetization}
%Let 
%\[
%Z := \frac{\beta h}{\sqrt{|B_n|}} \sum_{i\in B_n} J_i,
%\]
and define the random variable
\begin{align}\label{ldef}
L  = L(\sigma) := \frac{\beta h}{\sqrt{|B_n|}}\sum_{i\in B_n} J_i \sigma_i, %,  \ \ \ m = m(\sigma) := \frac{1}{|B_n|} \sum_{i\in B_n}  \sigma_i,
\end{align}
where $\sigma$ is drawn from the Ising model on $B_n$ at inverse temperature $\beta$ and free boundary condition. Let $m$ and $R_{1,2}$ be the magnetization of $\sigma$ and the overlap between two configurations drawn independently from the Gibbs measure of the Ising model, respectively. Let $\beta$ and $q$ be as in Theorem \ref{diffthm}. The first step in the proof of Theorem \ref{magthm} is the following lemma.
\begin{lmm}\label{maglmm1}
Let $L$ be as above. Then  
\[
\lim_{n\to \infty} \E[(\smallavg{e^L}_0 - e^{\frac{1}{2}\beta^2 h^2(1-q)} \cosh X_n )^2 ]= 0.
\]
\end{lmm}
\begin{proof}
Note that by Lemma \ref{gaussianlmm}, 
\begin{align}
\E\smallavg{e^L}_0^2 &= \E\smallavg{e^{L(\sigma^1) + L(\sigma^2)}}_0\notag \\
&= \E \bigavg{\exp\biggl(\frac{\beta h}{\sqrt{|B_n|}}\sum_{i\in B_n} J_i(\sigma_i^1 + \sigma_i^2)\biggr)}_0\notag \\
&=\bigavg{\exp\biggl(\frac{\beta^2h^2}{2|B_n|}\sum_{i\in B_n} (\sigma_i^1 + \sigma_i^2)^2\biggr) + o(1)}_0\notag \\
&= e^{\beta^2h^2} \smallavg{e^{\beta^2 h^2 R_{1,2}}}_0 + o(1). \label{submain1}
\end{align}
By Corollary \ref{diffcor}, this shows that
\begin{align}\label{main1}
\lim_{n\to \infty} \E\smallavg{e^L}_0^2 &= e^{\beta^2 h^2}\cosh(\beta^2 h^2 q). 
\end{align}
Next, again by Lemma \ref{gaussianlmm}, 
\begin{align}
e^{\beta^2 h^2(1-q)}  \E \cosh^2X_n &= \frac{1}{4}e^{\beta^2 h^2(1-q)} \E(e^{2X_n} + e^{-2X_n} + 2) \notag \\
&=  \frac{1}{2}e^{\beta^2 h^2(1-q)}(e^{2\beta^2 h^2 q} + 1) + o(1)\notag \\
&= e^{\beta^2 h^2} \cosh(\beta^2 h^2 q) + o(1). \label{main2}
\end{align}
Finally, by another application of Lemma \ref{gaussianlmm}, 
\begin{align}
&e^{\frac{1}{2}\beta^2 h^2(1-q)}\E[\smallavg{e^L}_0 \cosh X_n] = \frac{1}{2}e^{\frac{1}{2}\beta^2 h^2(1-q)} \E[\smallavg{e^{L + X_n}}_0 + \smallavg{e^{L - X_n}}_0]\notag \\
&=  \frac{1}{2}e^{\frac{1}{2}\beta^2 h^2(1-q)} \E\biggl[\bigavg{\exp\biggl(\frac{\beta h}{\sqrt{|B_n|}}\sum_{i\in B_n} J_i(\sigma_i + \sqrt{q})\biggr)}_0\notag \\
&\qquad \qquad + \bigavg{\exp\biggl(\frac{\beta h}{\sqrt{|B_n|}}\sum_{i\in B_n} J_i(\sigma_i -  \sqrt{q})\biggr)}_0\biggr]\notag \\ 
&= \frac{1}{2}e^{\frac{1}{2}\beta^2 h^2(1-q)} \biggl[\bigavg{\exp\biggl(\frac{\beta^2 h^2}{2|B_n|}\sum_{i\in B_n} (\sigma_i +  \sqrt{q})^2\biggr)}_0 \notag \\
&\qquad \qquad + \bigavg{\exp\biggl(\frac{\beta^2 h^2}{2|B_n|}\sum_{i\in B_n} (\sigma_i - \sqrt{q})^2\biggr)}_0 + o(1)\biggr]\notag \\
&= \frac{1}{2} e^{\beta^2 h^2} [\smallavg{e^{\beta^2 h^2\sqrt{q} m}}_0 + \smallavg{e^{-\beta^2 h^2\sqrt{q}m}}_0] + o(1). \label{submain3}
\end{align}
%Now note that $w$ has the same law as the magnetization in the ferromagnetic Ising model on $B_n$ at inverse temperature $\beta$ and free boundary condition. 
But, by Corollary \ref{diffcor},
\begin{align*}
\lim_{n\to \infty} \smallavg{e^{\beta^2 h^2\sqrt{q} m}}_0 = \lim_{n\to \infty} \smallavg{e^{- \beta^2 h^2\sqrt{q} m}}_0 = \cosh(\beta^2 h^2 q).
\end{align*}
Thus,
\begin{align}\label{main3}
\lim_{n\to \infty} e^{\frac{1}{2}\beta^2 h^2(1-q)}\E[\smallavg{e^L}_0 \cosh X_n]  = e^{\beta^2 h^2} \cosh(\beta^2 h^2 q). 
\end{align}
Combining \eqref{main1}, \eqref{main2} and \eqref{main3}, we get
\begin{align*}
&\lim_{n\to \infty} \E[(\smallavg{e^L}_0 - e^{\frac{1}{2}\beta^2 h^2(1-q)} \cosh X_n)^2 ] \\
&= \lim_{n\to \infty} \E[\smallavg{e^L}_0^2 - 2 e^{\frac{1}{2}\beta^2 h^2(1-q)} \smallavg{e^L}_0\cosh X_n +  e^{\beta^2 h^2(1-q)} \cosh^2 X_n] = 0.
\end{align*}
This completes the proof of the lemma.
\end{proof}

The next step in the proof of Theorem \ref{magthm} is the following lemma.
\begin{lmm}\label{maglmm2}
Let $L$ be as above. Then 
\[
\lim_{n\to \infty} \E[(\smallavg{me^L}_0 - \sqrt{q}e^{\frac{1}{2}\beta^2h^2(1-q)}\sinh X_n)^2] = 0.
\]
\end{lmm}
\begin{proof}
Take any $j\in B_n$. By a computation similar to the one that led to equation \eqref{submain1}, we get
\begin{align*}
\E\smallavg{\sigma_j e^L}_0^2 &= e^{\beta^2 h^2} \smallavg{\sigma_j^1 \sigma_j^2 e^{\beta^2 h^2R_{1,2}}}_0 + o(1). 
\end{align*}
Averaging over $j$, we get
\begin{align*}
\frac{1}{|B_n|} \sum_{j\in B_n} \E\smallavg{\sigma_j e^L}_0^2 &= e^{\beta^2 h^2} \smallavg{R_{1,2} e^{\beta^2 h^2R_{1,2}}}_0 + o(1).
\end{align*}
By Corollary \ref{diffcor}, this shows that
\begin{align}\label{main21}
\lim_{n\to\infty} \frac{1}{|B_n|} \sum_{j\in B_n} \E\smallavg{\sigma_j e^L}_0^2 &= q e^{\beta^2 h^2}\sinh(\beta^2 h^2q). 
\end{align}
Next, by a computation similar to the one that led to equation \eqref{main2}, we get
\begin{align}\label{main22}
qe^{\beta^2h^2(1-q)}\E\sinh^2 X_n &= q e^{\beta^2 h^2}\sinh(\beta^2h^2 q) + o(1).
\end{align}
Finally, by a computation similar to the one that led to equation \eqref{submain3}, we get
\begin{align*}
&\sqrt{q}e^{\frac{1}{2}\beta^2h^2(1-q)}\E[\smallavg{\sigma_j e^L}_0 \sinh X_n] \\
&= \frac{\sqrt{q}}{2} e^{\beta^2h^2}[\smallavg{\sigma_je^{\beta^2 h^2\sqrt{q} m}}_0 - \smallavg{\sigma_je^{-\beta^2 h^2\sqrt{q}m}}_0] + o(1).
%&= \sqrt{q} e^{\beta^2h^2} \sinh(
\end{align*}
Averaging this over $j$, we get
\begin{align*}
\frac{1}{|B_n|}\sum_{j\in B_n}\sqrt{q} e^{\frac{1}{2}\beta^2 h^2(1-q)}\E[\smallavg{\sigma_j e^L}_0 \sinh X_n] &= \sqrt{q} e^{\beta^2 h^2} \smallavg{m\sinh(\beta^2 h^2 \sqrt{q} m)}_0 + o(1). 
\end{align*}
By Corollary \ref{diffcor}, this shows that
\begin{align}\label{main23}
\lim_{n\to\infty} \frac{1}{|B_n|}\sum_{j\in B_n}\sqrt{q}e^{\frac{1}{2}\beta^2 h^2(1-q)}\E[\smallavg{\sigma_j e^L}_0 \sinh X_n] &= qe^{\beta^2 h^2} \sinh(\beta^2 h^2 q). 
\end{align}
Combining \eqref{main21}, \eqref{main22} and \eqref{main23}, we get
\begin{align}\label{neweq}
\lim_{n\to\infty} \frac{1}{|B_n|} \sum_{j\in B_n} \E[(\smallavg{\sigma_j e^L}_0 - \sqrt{q}e^{\frac{1}{2}\beta^2h^2(1-q)}\sinh X_n)^2] = 0.
\end{align}
An application of the Cauchy--Schwarz inequality shows that the above quantity is an upper bound for the quantity that we want to show is converging to zero, thereby completing the proof.
%Since $\E\smallavg{e^L}^2$ and $\E\sinh^2(\sqrt{q}Z)$ are uniformly bounded in $n$, an application of the Cauchy--Schwarz inequality along with the above identity shows that
%\begin{align*}
%\lim_{n\to\infty} \frac{1}{|B_n|} \sum_{j\in B_n} \E|\smallavg{(-1)^{|j|_1}\sigma_j e^L}^2 - q e^{\beta^2 h^2(1-q)}\sinh^2(\sqrt{q} Z)| = 0.
%\end{align*}
%Since 
%\[
%\frac{1}{|B_n|} \sum_{j\in B_n} \smallavg{(-1)^{|j|_1}\sigma_j e^L}^2 = \smallavg{R_{1,2} e^{L(\sigma^1) + L(\sigma^2)}}, 
%\]
%the previous display implies that 
%\[
%\lim_{n\to\infty} \E|\smallavg{R_{1,2} e^{L(\sigma^1) + L(\sigma^2)}} - q e^{\beta^2h^2(1-q)}\sinh^2(\sqrt{q} Z)| = 0.
%\]
\end{proof}

The final ingredient we need is the following.
\begin{lmm}\label{maglmm3}
Let $\beta$ be as in Theorem \ref{diffthm}. Then  $\smallavg{(R_{1,2} - m(\sigma^1)m(\sigma^2))^2}_0 \to 0$ as $n\to\infty$, where $\sigma^1$ and $\sigma^2$ are drawn independently from the Ising model on $B_n$ at inverse temperature $\beta$ and free boundary condition. 
\end{lmm}
\begin{proof}
Note that 
\begin{align*}
\smallavg{(R_{1,2} - m(\sigma^1)m(\sigma^2))^2}_0 &= \smallavg{R_{1,2}^2}_0 - 2\smallavg{R_{1,2}m(\sigma^1)m(\sigma^2)}_0 + \smallavg{m^2}_0^2. 
\end{align*}
By Corollary \ref{diffcor}, $\smallavg{R_{1,2}^2}_0 \to q^2$ and $\smallavg{m^2}_0 \to q$ as $n\to \infty$. Now, note that
\begin{align*}
\smallavg{R_{1,2}m(\sigma^1)m(\sigma^2)}_0 &= \bigavg{\frac{1}{|B_n|^3}\sum_{i,j,k\in B_n} \sigma_i^1 \sigma_i^2 \sigma_j^1 \sigma_k^2}_0\\
&= \frac{1}{|B_n|^3} \sum_{i,j,k\in B_n} \smallavg{\sigma_i\sigma_j}_0 \smallavg{\sigma_i\sigma_k}_0. 
\end{align*}
Using the same tactics as in the proof of Theorem \ref{diffthm2}, it is now easy to show that the above quantity tends to $q^2$ as $n\to \infty$. This completes the proof. 
\end{proof}

We are now ready to complete the proof of Theorem \ref{magthm}.
\begin{proof}[Proof of Theorem \ref{magthm}]
 First, note that by Jensen's inequality,
\begin{align}\label{jensen}
\smallavg{e^L}_0 &\ge e^{\smallavg{L}_0} = 1. 
\end{align}
Thus, by \eqref{jensen} and the Cauchy--Schwarz inequality,  
\begin{align*}
\smallavg{|m^2 -q|} &= \frac{\smallavg{|m^2 -q| e^L}_0}{\smallavg{e^L}_0} \\
&\le \smallavg{|m^2 -q| e^L}_0 \le \sqrt{\smallavg{(m^2-q)^2}_0\smallavg{e^{2L}}_0}. 
\end{align*}
Taking expectation on both sides and applying Jensen's inequality, we get
\[
\ee(\smallavg{|m^2 -q|}) \le \sqrt{\smallavg{(m^2-q)^2}_0\ee(\smallavg{e^{2L}}_0)}.
\]
By Theorem \ref{diffthm2}, the first term within the square-root tends to zero as $n\to\infty$. By Lemma \ref{gaussianlmm}, the second term is uniformly bounded in $n$. Thus, $\ee(\smallavg{|m^2 -q|}) \to 0$ as $n\to\infty$. Since 
\[
\smallavg{(m^2 -q)^2} \le 2\smallavg{|m^2 -q|},
\]
this proves the first claim of the theorem. Next, note that 
\begin{align*}
\smallavg{m} &= \frac{\smallavg{me^L}_0}{\smallavg{e^L}_0}. 
\end{align*}
Thus, if we let
\[
a := \sqrt{q} e^{\frac{1}{2}\beta^2h^2(1-q)}\sinh X_n, \ \ b := e^{\frac{1}{2}\beta^2h^2(1-q)} \cosh X_n, \ \ c := \frac{a}{b} = \sqrt{q} \tanh X_n,
\]
then by \eqref{jensen}, 
\begin{align*}
|\smallavg{m}  - c| &= \biggl| \frac{\smallavg{me^L}_0}{\smallavg{e^L}_0} - \frac{a}{b}\biggr|= \frac{|b \smallavg{me^L}_0 - a\smallavg{e^L}_0|}{b \smallavg{e^L}_0} \\
&\le\frac{ |b \smallavg{me^L}_0 - a\smallavg{e^L}_0|}{b} \le |\smallavg{me^L}_0 - a| +\frac{ a}{b} |\smallavg{e^L}_0-b|. 
\end{align*}
Since $b\ge 1$, this shows that
\begin{align}
\E|\smallavg{m} - c| &\le \E|\smallavg{me^L}_0 - a| + \E(a |\smallavg{e^L}_0-b|)\notag \\
&\le \sqrt{\E[(\smallavg{me^L}_0 - a)^2]} + \sqrt{\E(a^2) \E[(\smallavg{e^L}_0-b)^2]}.\label{newderivation}
\end{align}
By Lemma \ref{maglmm1}, Lemma \ref{maglmm2}, and the fact that $\E(a^2)$ is uniformly bounded in $n$ (by Lemma \ref{gaussianlmm}), we get that the above quantity tends to zero as $n\to\infty$. But, since $\smallavg{m}$ and $c$ are both in $[-1,1]$, 
\[
\E[(\smallavg{m} - c)^2] \le 2 \E|\smallavg{m} - c|.
\]
This proves \eqref{mform}. To prove \eqref{mformgen}, note that by Lemma \ref{maglmm1} and the inequality \eqref{jensen},  proceeding as in the derivation of \eqref{newderivation}, we get
\begin{align*}
&\frac{1}{|B_n|} \sum_{j\in B_n} \E|\smallavg{\sigma_j} - \sqrt{q}\tanh X_n| = \frac{1}{|B_n|} \sum_{j\in B_n} \E\biggl|\frac{\smallavg{\sigma_je^L}_0}{\smallavg{e^L}_0} - \frac{a}{b}\biggr|\\
&= \frac{1}{|B_n|} \sum_{j\in B_n} \E\biggl(\frac{|b \smallavg{\sigma_je^L}_0 - a \smallavg{e^L}_0|}{b\smallavg{e^L}_0}\biggr)\\
&\le  \frac{1}{|B_n|} \sum_{j\in B_n} \E\biggl(\frac{|b \smallavg{\sigma_je^L}_0 - a \smallavg{e^L}_0|}{b}\biggr)\\
&\le  \sqrt{\E(a^2) \E[(\smallavg{e^L}_0-b)^2]} + \frac{1}{|B_n|} \sum_{j\in B_n}\sqrt{\E[(\smallavg{\sigma_je^L}_0 - a)^2]}.
\end{align*}
We have already seen that the first term tends to zero as $n\to\infty$. The second term is bounded above by 
\begin{align*}
\biggl[\frac{1}{|B_n|} \sum_{j\in B_n}\E[(\smallavg{\sigma_je^L}_0 - a)^2]\biggr]^{\frac{1}{2}}.
\end{align*}
By \eqref{neweq}, this also tends to zero as $n\to\infty$. Thus, 
\[
\lim_{n\to\infty}\frac{1}{|B_n|} \sum_{j\in B_n} \E|\smallavg{\sigma_j} - \sqrt{q}\tanh X_n| = 0. 
\]
Thus, by \eqref{mform} and the fact that
\[
(\smallavg{\sigma_j} - \sqrt{q}\tanh X_n)^2\le 2|\smallavg{\sigma_j} - \sqrt{q}\tanh X_n|,
\]
we get \eqref{mformgen}. Finally, note that by \eqref{jensen} and the Cauchy--Schwarz inequality,
\begin{align*}
\E\smallavg{|R_{1,2} - m(\sigma^1) m(\sigma^2)|} &= \E\biggl(\frac{\smallavg{|R_{1,2} - m(\sigma^1) m(\sigma^2)|e^{L(\sigma^1) + L(\sigma^2)}}_0}{\smallavg{e^L}_0^2}\biggr)\\
&\le \E\smallavg{|R_{1,2} - m(\sigma^1) m(\sigma^2)|e^{L(\sigma^1) + L(\sigma^2)}}_0\\
&\le \sqrt{\smallavg{(R_{1,2} - m(\sigma^1) m(\sigma^2))^2}_0\E\smallavg{e^{2L(\sigma^1) + 2L(\sigma^2)}}_0}
\end{align*}
By Lemma \ref{maglmm3}, the first term inside the square-root tends to zero as $n\to\infty$. The second term is uniformly bounded in $n$, by Lemma \ref{gaussianlmm}. This shows that the expression on the left tends to zero. But note that
\[
\E\smallavg{(R_{1,2} - m(\sigma^1) m(\sigma^2))^2} \le 2\E\smallavg{|R_{1,2} - m(\sigma^1) m(\sigma^2)|}. 
\]
This completes the proof of Theorem \ref{magthm}.
\end{proof}

\subsection{Proof of Theorem \ref{rsbthm}}
All assertions of Theorem \ref{rsbthm} are direct consequences of the properties of $m$ from Theorem \ref{magthm} and the result that $\E\smallavg{(R_{1,2} - m(\sigma^1) m(\sigma^2))^2} \to 0$ as $n\to\infty$.

\subsection{Proof of Theorem \ref{ultrathm}}
Let $A := \{-\sqrt{q},\sqrt{q}\}^3 \subseteq \R^3$, and let $B$ denote the set displayed in \eqref{qqq}. Consider the map $f: \R^3 \to \R^3$ defined as $f(x,y,z) := (xy, yz, zx)$. Then $f$ is a continuous map, and an easy verification shows that $f(A) = B$. (For example, $f(\sqrt{q},\sqrt{q},\sqrt{q}) = (q,q,q)$, $f(\sqrt{q},\sqrt{q},-\sqrt{q}) = (q, -q, -q)$, $f(\sqrt{q},-\sqrt{q},-\sqrt{q}) = (-q,q,-q)$, etc.) Take any open set $V\supseteq B$, and let $U=f^{-1}(V)$. Then $U$ is also open, and $U\supseteq A$. Let $\sigma^1, \sigma^2, \sigma^3$ be three configurations drawn independently from the Gibbs measure of our model, and define the overlaps as usual. By Theorem \ref{magthm}, the difference between the random vectors $(R_{1,2}, R_{2,3}, R_{3,1})$ and $f(m(\sigma^1), m(\sigma^2), m(\sigma^3))$ converges to the zero vector in $L^2$ (unconditionally, after integrating out the disorder).  This shows, first of all, that the quenched law of $(R_{1,2}, R_{2,3}, R_{3,1})$ converges in distribution, because so does the quenched law of $(m(\sigma^1), m(\sigma^2), m(\sigma^3))$. Next, note that by Theorem~\ref{magthm},
\[
\lim_{n\to\infty} \P((m(\sigma^1), m(\sigma^2), m(\sigma^3)) \in U) =1,
\]
where $\P$ denotes the unconditional probability, after integrating out the disorder. Thus,
\[
\lim_{n\to\infty} \P(f(m(\sigma^1), m(\sigma^2), m(\sigma^3)) \in V) =1. 
\]
Combining this with the previous observation, we see that for any open set $V\supseteq B$,
\[
\lim_{n\to\infty} \P((R_{1,2}, R_{2,3}, R_{3,1}) \in V) =1.
\]
This shows that the quenched probability of the event $(R_{1,2}, R_{2,3}, R_{3,1}) \in V$ converges to $1$ in probability. From this, it is easy to complete the proof of the theorem.

\subsection{Proof of Theorem \ref{ggthm}}
Taking $k=2$, $f= S_{1,2}$ and $\psi(x) = x$ in \eqref{ggid} gives the equation
\[
\E(S_{1,2}S_{1,3}) = \frac{1}{2}(\E(S_{1,2}))^2 + \frac{1}{2}\E(S_{1,2}^2).
\]
We will show that this equation fails for the overlap in the infinite volume limit of our model. Indeed, by Theorem \ref{magthm},
\begin{align}
&\lim_{n\to\infty} (2\E\smallavg{R_{1,2}R_{1,3}} - (\E\smallavg{R_{1,2}})^2 - \E\smallavg{R_{1,2}^2}) \notag \\
&= \lim_{n\to\infty} (2\E\smallavg{m(\sigma^1)^2m(\sigma^2)m(\sigma^3)} - (\E\smallavg{m(\sigma^1) m(\sigma^2)})^2 - \E\smallavg{m(\sigma^1)^2m(\sigma^2)^2})\notag \\
&=\lim_{n\to\infty} (2\E(\smallavg{m^2} \smallavg{m}^2) - (\E(\smallavg{m}^2))^2 - \E(\smallavg{m^2}^2)),\label{rmeq}
\end{align}
provided that the limits exist (which we will prove shortly). Let $X_n$ be defined as in \eqref{xndef}, and let $Y_n := \tanh X_n$. 
%\[
%X_n := \sqrt{q}\tanh\biggl(\frac{\sqrt{q}\beta h}{\sqrt{|B_n|}}\sum_{i\in B_n} J_i \biggr). 
%\]
Then by Theorem~\ref{magthm}, the right side of \eqref{rmeq} equals
\begin{align*}
\lim_{n\to\infty} (2q^2\E(Y_n^2) - q^2(\E(Y_n^2))^2 - q^2) = - q^2 \lim_{n\to\infty} (1-\E(Y_n^2))^2.
 %\lim_{n\to\infty} (2\E(Y_n^4) - (\E(Y_n^2))^2 - \E(Y_n^4)) = \lim_{n\to\infty} \var(Y_n^2). 
\end{align*}
But $Y_n$ is a bounded random variable with converges in distribution to $\tanh(\sqrt{q}\beta h Z)$, where $Z$ is a standard Gaussian random variable. Thus, for any finite $h$, the above limit is nonzero. This completes the proof.

\subsection{Proof of Theorem \ref{antithm}}

%\begin{thm}\label{magthm0}
%Take any $d\ge 2$ and $\beta >0$. Let $m$ be the magnetization of the antiferromagnetic Ising model on $B_n$ under free boundary condition and inverse temperature $\beta$. Then for sufficiently large $\beta$,  $\smallavg{m^2}\to 0$ as $n\to \infty$.
%\end{thm}
%\begin{proof}
%In this proof, we will revert back to using $\smallavg{\cdot}$ for averaging with respect to our model, instead of the Ising model. 
In this subsection, we will denote averaging with respect to the antiferromagnetic Ising model on $B_n$ at inverse temperature $\beta$ and free boundary condition by $\smallavg{\cdot}_{a,0}$, and averaging with respect to the model on $B_n$ with Hamiltonian \eqref{antihamil} at inverse temperature $\beta$ by $\smallavg{\cdot}_a$. 

Let $\sigma$ be a configuration drawn from the ferromagnetic Ising model on $B_n$ at inverse temperature $\beta$ and free boundary condition. Define $\eta\in \Sigma_n$ as 
\begin{align}\label{sigmaeta}
\eta_i := (-1)^{|i|_1} \sigma_i \  \text{ for all $i\in B_n$.}
\end{align}
Then, it is easy to see that $\eta$ is drawn from antiferromagnetic Ising model on $B_n$ at inverse temperature $\beta$ and free boundary condition. Thus, we have 
\begin{align*}
\smallavg{m^2}_{a,0} &= \frac{1}{|B_n|^2}\sum_{i,j\in B_n } \smallavg{\sigma_i \sigma_j}_{a,0} = \frac{1}{|B_n|^2}\sum_{i,j\in B_n } (-1)^{|i|_1+|j|_1}\smallavg{\sigma_i \sigma_j}_0.
\end{align*}
Take any $\ve \in (0,1)$. Let $\delta_n$ be as in Theorem \ref{diffthm} and let $m := \lfloor (1-\ve)n \rfloor$.  Let  $S$ be defined as in equation \eqref{sdefin}, 
and let $S^c := (B_n \times B_n)\setminus S$. 
Then 
\begin{align*}
\frac{1}{|B_n|^2}\sum_{i,j\in B_n } (-1)^{|i|_1+|j|_1}(\smallavg{\sigma_i \sigma_j}_0-q) &\le \frac{|S^c|}{|B_n|^2} + \frac{|S|\delta_n}{|B_n|^2}\le \frac{|S^c|}{|B_n|^2} + \delta_n. 
\end{align*}
By Theorem \ref{diffthm}, $\delta_n \to 0$ as $n\to \infty$. Combining these observations and the upper bound \eqref{scbound}, we get
\begin{align*}
\limsup_{n\to\infty} \frac{1}{|B_n|^2}\sum_{i,j\in B_n } (-1)^{|i|_1+|j|_1}(\smallavg{\sigma_i \sigma_j}_0-q)&\le C\ve. 
\end{align*}
It is easy to see that 
\[
\lim_{n\to\infty} \frac{1}{|B_n|^2}\sum_{i,j\in B_n } (-1)^{|i|_1+|j|_1} = 0. 
\]
Combining all of the above, we get that $\limsup_{n\to\infty} \smallavg{m^2}_{a,0}  \le C\ve$. Since this holds for every $\ve \in (0,1)$, and $\smallavg{m^2}_{a,0}\ge0$, we conclude that $\smallavg{m^2}_{a,0}\to 0$ as $n\to \infty$. 

Now let $L$ be defined as in \eqref{ldef}. Then, as in \eqref{jensen}, we have $\smallavg{e^L}_{a,0} \ge e^{\smallavg{L}_{a,0}} = 1$. Thus, 
\begin{align*}
\E\smallavg{|m|}_a &= \E \biggl(\frac{\smallavg{|m| e^L}_{a,0}}{\smallavg{e^L}_{a,0}}\biggr) \le\E \smallavg{|m| e^L}_{a,0}\le \sqrt{\E\smallavg{m^2}_{a,0} \E \smallavg{e^L}_{a,0}}.
\end{align*}
We have shown above that the first term inside the square-root tends to zero as $n\to\infty$. By Lemma \ref{gaussianlmm} and the above relationship between $\eta$ and $\sigma$, the second term is uniformly bounded in $n$. Thus, $\E\smallavg{|m|}_a \to 0 $ as $n\to\infty$. Since $|m|\le 1$, this implies that $\E\smallavg{m^2}_a \to 0$.

Lastly, let $\sigma^1$ and $\sigma^2$ are configurations drawn independently from the model on $B_n$ with Hamiltonian \eqref{hamil} at inverse temperature $\beta$, but with $J_i$ replaced by $(-1)^{|i|_1} J_i$. Define $\eta^1$ and $\eta^2$ via the relationship \eqref{sigmaeta}. 
Then, it is easy to see that $\eta^1$ and $\eta^2$ are drawn independently from the model on $B_n$ with Hamiltonian \eqref{antihamil} at inverse temperature $\beta$. Moreover, the overlap between $\eta^1$ and $\eta^2$ is exactly the same as the overlap between $\sigma^1$ and $\sigma^2$. Thus, all of the claims about the overlap that we have proved for the ferromagnetic model continue to hold for the antiferromagnetic model, after replacing $J_i$ by $(-1)^{|i|_1}J_i$ in the theorem statements. 

%, because %$J_i \eta_i = (-1)^{|i|_1}$ for any neighboring $i$ and $j$, 
%\[
%\sum_{\{i,j\}\in E_n} \eta_i\eta_j - \sum_{i\in B_n} J_i \eta_i = -\sum_{\{i,j\}\in B_n} \sigma_i \sigma_j - \sum_{i\in B_n} (-1)^{|i|_1} J_i\sigma_i. 
%\]
 
%Lastly, note that if $\sigma^1$ and $\sigma^2$ are independent draws from the ferromagnetic model and $\eta^1$ and $\eta^2$ are defined as above, then

\subsection{Proof of Theorem \ref{negthm}}
Let $F$ denote the free energy of our model. That is,
\[
F = \log \sum_{\sigma\in \Sigma_n} e^{-\beta H_n(\sigma)},
\]
with $H_n$ defined as in \eqref{hamil}. Then note that 
\begin{align*}
\fpar{F}{J_i} &= \frac{\beta h\smallavg{\sigma_i}}{\sqrt{|B_n|}}.
\end{align*}
This implies, by the Gaussian Poincar\'e inequality~\cite[p.~49]{ledoux01}, that 
\begin{align}\label{varup}
\var(F) \le \sum_{i\in B_n} \E\biggl[\biggl(\fpar{F}{J_i}\biggr)^2\biggr] \le \beta^2 h^2.
\end{align}
On the other hand,  
\[
\mpar{F}{J_i}{J_j} = \frac{\beta^2h^2}{|B_n|} (\sij - \si \sj),
\]
and therefore, by \cite[Theorem 3.1]{chatterjee18}, 
\begin{align}
\var(F) &\ge \frac{1}{2}\sum_{i,j\in B_n} \biggl[\E\biggl(\mpar{F}{J_i}{J_j}\biggr)\biggr]^2\notag \\
&= \frac{\beta^4h^4}{2|B_n|^2} \sum_{i,j\in B_n} [\E(\sij-\si\sj)]^2. \label{varlow}
\end{align}
By the FKG inequality for the RFIM~\cite[Lemma 2.5]{chatterjee15}, $\sij-\si\sj \ge 0$ for $i, j$. Thus,
\begin{align*}
\smallavg{R_{1,2}^2} - \smallavg{R_{1,2}}^2 &= \frac{1}{|B_n|^2} \sum_{i,j\in B_n} (\sij^2 - \si^2\sj^2)\\
&= \frac{1}{|B_n|^2} \sum_{i,j\in B_n} (\sij - \si\sj)(\sij+\si\sj)\\
&\le \frac{2}{|B_n|^2} \sum_{i,j\in B_n} |\sij - \si\sj| \\
&= \frac{2}{|B_n|^2} \sum_{i,j\in B_n} (\sij - \si\sj). 
\end{align*}
Combining this with \eqref{varlow}, we get
\begin{align*}
\E\smallavg{(R_{1,2} - \smallavg{R_{1,2}})^2} &= \E(\smallavg{R_{1,2}^2} - \smallavg{R_{1,2}}^2)\\
&\le \frac{2}{|B_n|^2} \sum_{i,j\in B_n} \E(\sij - \si\sj)\\
&\le 2\biggl[\frac{1}{|B_n|^2} \sum_{i,j\in B_n} (\E(\sij - \si\sj))^2\biggr]^{\frac{1}{2}}\\
&\le 2\biggl[\frac{2\var(F)}{\beta^4h^4}\biggr]^{\frac{1}{2}}. 
\end{align*}
Plugging in the upper bound on $\var(F)$ from \eqref{varup}, we get
\begin{align*}
\E\smallavg{(R_{1,2} - \smallavg{R_{1,2}})^2} &\le \frac{2^{\frac{3}{2}}}{\beta |h|}. 
\end{align*}
The upper bound tends to zero if $|h|\to\infty$ as $n\to \infty$. This proves the second claim of Theorem~\ref{negthm}. For the first claim, note that for any $k$,
\begin{align*}
\smallavg{R_{1,2}^k} &= \frac{\smallavg{R_{1,2}^k e^L}_0}{\smallavg{e^L}_0},
\end{align*}
where $L$ is the function defined in \eqref{ldef}. 
By \eqref{jensen}, this shows that
\begin{align*}
|\smallavg{R_{1,2}^k} - \smallavg{R_{1,2}^k}_0| &= \biggl|\frac{\smallavg{R_{1,2}^k e^L}_0}{\smallavg{e^L}_0} - \smallavg{R_{1,2}^k}_0\biggr|\\
&= \frac{|\smallavg{R_{1,2}^k e^L}_0 - \smallavg{R_{1,2}^k}_0\smallavg{e^L}_0|}{\smallavg{e^L}_0}\\
&\le |\smallavg{R_{1,2}^k e^L}_0 - \smallavg{R_{1,2}^k}_0\smallavg{e^L}_0|\\
&= |\smallavg{R_{1,2}^k (e^L-\smallavg{e^L}_0)}_0|\\
&\le \smallavg{|e^L - \smallavg{e^L}_0|}_0 \\
&=  \smallavg{|(e^L-1) - \smallavg{e^L-1}_0|}_0 \le 2\smallavg{|e^L-1|}_0.
\end{align*}
Now, note that for any given $\sigma\in \Sigma_n$, by Lemma \ref{gaussianlmm}, 
\begin{align*}
\E|e^{L(\sigma)} - 1| &\le \sqrt{\E[(e^{L(\sigma)} - 1)^2]}\\
&= \sqrt{\E[e^{2L(\sigma)} - 2e^{L(\sigma)} + 1]}\\
&= \sqrt{e^{2\beta^2 h^2} - 2e^{\frac{1}{2}\beta^2h^2} + 1 + o(1)}\\
&= \sqrt{(e^{2\beta^2 h^2} - 1) - 2(e^{\frac{1}{2}\beta^2h^2}  - 1) + o(1)}. 
\end{align*}
By the inequality $e^x - 1\le ex$ that holds for $0\le x\le 1$, we get that the above quantity is bounded above by $C\beta|h| + o(1)$ when $|h|\le \frac{1}{\beta}$, where $C$ is a universal constant. In particular, it tends to zero as $h\to 0$.  Thus, if $h\to 0$ as $n\to\infty$, then for every $k$,
\begin{align*}
\lim_{n\to\infty} \E|\smallavg{R_{1,2}^k}- \smallavg{R_{1,2}^k}_0| = 0.
\end{align*}
This shows that if $n$ is large, then all quenched moments of $R_{1,2}$ under our model are, with high probability, close to the corresponding moments of $R_{1,2}$ under the Ising model. From this, it is not hard to prove the claim stated in the theorem (e.g., using Bernstein approximation).

\subsection{Proof of Theorem \ref{purethm}}
By Theorem \ref{magthm} and Theorem \ref{rsbthm}, $\E\smallavg{(R_{1,2}^2 - q^2)^2} \to 0$ and $\E\smallavg{(m^2 - q)^2} \to 0$ for our model. This allows us to repeat the proof of Theorem \ref{diffthmmain} from Subsection \ref{diffproofsec} verbatim to deduce that the conclusion of Theorem \ref{diffthmmain} holds even if $h\ne 0$, with the same $q$. This shows, in particular, that if $\pi_n$ is a uniform random permutation of the elements of $B_n$, then for any even $l$, 
\begin{align}\label{evenl}
\smallavg{\sigma_{\pi_n(1)}\sigma_{\pi_n(2)}\cdots \sigma_{\pi_n(l)}} \to q^{\frac{l}{2}}
\end{align}
in probability as $n\to\infty$. Next, let us consider the case of odd $l$. For the Ising model, the above expectation is zero if $l$ is odd. This is no longer true if $h\ne 0$. Recall the random variable $X_n$ defined in equation \eqref{xndef}. Take any odd positive integer $l$. We claim that 
\begin{align}\label{oddl}
\smallavg{\sigma_{\pi_n(1)}\sigma_{\pi_n(2)}\cdots \sigma_{\pi_n(l)}} - q^{\frac{l}{2}}\tanh X_n \to 0
\end{align}
in probability as $n\to\infty$.  To prove this, let $m$ be the magnetization. We claim that 
\begin{align}\label{mlim}
\smallavg{m^l} - q^{\frac{l}{2}}\tanh X_n \to 0
\end{align}
in probability as $n\to\infty$. To see this, note that by Theorem \ref{magthm}, $\E\smallavg{(m^2 - q)^2} \to 0$. From this, it is easy to deduce that
\[
\smallavg{m^l} - q^{\frac{1}{2}(l-1)} \smallavg{m} \to 0
\]
in probability, since $m^{l-1}$ can be replaced by $q^{\frac{1}{2}(l-1)}$ asymptotically. But again by Theorem~\ref{magthm}, $\smallavg{m} - \sqrt{q}\tanh X_n \to 0$ in probability. Combining these two observations yields \eqref{mlim}. Next, we claim that 
\begin{align}\label{rlimit}
\smallavg{R_{1,2}^l} - q^l \tanh^2 X_n \to 0
\end{align}
in probability as $n\to\infty$. To see this, note that by Theorem \ref{magthm}, $\smallavg{(R_{1,2}-m(\sigma^1)m(\sigma^2))^2} \to 0$ in probability. This implies that 
\[
\smallavg{R_{1,2}^l} - \smallavg{m(\sigma^1)^l m(\sigma^2)^l} \to 0
\]
in probability. But $\smallavg{m(\sigma^1)^l m(\sigma^2)^l} = \smallavg{m^l}^2$.  Thus, \eqref{rlimit} follows from \eqref{mlim}. Now, proceeding just as in the derivation of \eqref{r12l}, we get
\begin{align*}
&\frac{1}{|B_n|^l}\sum_{i_1,\ldots,i_l\in B_n} |\smallavg{\sigma_{i_1}\cdots \sigma_{i_l}} - q^{\frac{l}{2}}\tanh X_n|  \\
&\le \biggl[\frac{1}{|B_n|^l}\sum_{i_1,\ldots,i_l\in B_n} (\smallavg{\sigma_{i_1}\cdots \sigma_{i_l}} - q^{\frac{l}{2}}\tanh X_n)^2\biggr]^{\frac{1}{2}}\\
&= \biggl[\frac{1}{|B_n|^l}\sum_{i_1,\ldots,i_l\in B_n} (\smallavg{\sigma_{i_1}\cdots \sigma_{i_l}}^2 - 2q^{\frac{l}{2}}\smallavg{\sigma_{i_1}\cdots \sigma_{i_l}} \tanh X_n + q^l\tanh^2 X_n)\biggr]^{\frac{1}{2}} \\
&= [\smallavg{R_{1,2}^l} - 2q^{\frac{l}{2}}\smallavg{m^l}\tanh X_n + q^l\tanh^2 X_n]^{\frac{1}{2}}. %\label{r12l2}
\end{align*}
By \eqref{mlim} and \eqref{rlimit}, the last expression tends to zero in probability as $n\to\infty$. This proves the claim \eqref{oddl}.

Now take any $n$, and let $\tau_{n,1},\tau_{n,2},\ldots$ be an infinite exchangeable sequence of random variables with the following random law. Given $X_n$, let $Z_n$ be a random variable that takes value $\sqrt{q}$ with probability $\frac{1}{2}(1+\tanh X_n)$ and $-\sqrt{q}$ with probability $\frac{1}{2}(1-\tanh X_n)$. Having generated $Z_n$, let $\tau_{n,1},\tau_{n,2},\ldots$ be i.i.d.~random variables taking value $1$ with probability $\frac{1}{2}(1+Z_n)$ and $-1$ with probability $\frac{1}{2}(1-Z_n)$. Then note that $\E(\tau_{n,i}|Z_n, X_n) = Z_n$, and therefore, for any positive integer~$l$, 
\[
\E(\tau_{n,1}\cdots \tau_{n,l}|Z_n, X_n) = Z_n^l.
\]
This give us
\[
\E(\tau_{n,1}\cdots \tau_{n,l}|X_n) = \E(Z_n^l|X_n) = 
\begin{cases}
q^{\frac{l}{2}}\tanh X_n &\text{ if $l$ is odd,}\\
q^{\frac{l}{2}} &\text{ if $l$ is even.} 
\end{cases}
\]
Comparing this with \eqref{evenl} and \eqref{oddl}, it is now easy to show that for any $l$, the L\'evy--Prokhorov distance between the (random) laws of $(\sigma_{\pi_n(1)},\ldots, \sigma_{\pi_n(l)})$ and $(\tau_{n,1},\ldots, \tau_{n,l})$ converges to zero in probability as $n\to\infty$. But, the random law of $(\tau_{n,1},\ldots, \tau_{n,l})$  converges in distribution to the random law of $(\tau_1,\ldots,\tau_l)$, where $\tau_1,\tau_2,\ldots$ are defined just like the $\tau_{n,i}$'s, but with $X_n$ replaced by $X = \sqrt{q}\beta h W$, where $W$ is a standard Gaussian random variable. This suffices to complete the proof.

\bibliographystyle{abbrvnat}

\bibliography{myrefs}

\end{document}